\documentclass[11pt,a4paper]{article}
\usepackage[T1]{fontenc}
\usepackage[utf8]{inputenc}
\usepackage{newtxtext}
\usepackage[margin=26mm,headheight=14pt]{geometry}
\usepackage{amsmath,amsthm,mathtools}
\usepackage{newtxmath}
\usepackage{booktabs,array,tabularx,longtable,float,needspace}
\usepackage{microtype,fancyhdr,enumitem}
\usepackage[hidelinks]{hyperref}
\usepackage{aliascnt}
\usepackage[nameinlink,noabbrev]{cleveref}
\hypersetup{pdftitle={Positive Sectional Curvature on Homotopy Nine-Spheres},pdfauthor={Zhiqi Chen, Shaoqiang Deng, Junna Ni, Hui Zhang},pdfsubject={Homotopy nine-spheres, smooth structures, and positive sectional curvature}}
\allowdisplaybreaks[2]
\newtheorem{theorem}{Theorem}[section]
\newaliascnt{proposition}{theorem}
\newtheorem{proposition}[proposition]{Proposition}
\aliascntresetthe{proposition}
\newaliascnt{lemma}{theorem}
\newtheorem{lemma}[lemma]{Lemma}
\aliascntresetthe{lemma}
\newaliascnt{corollary}{theorem}
\newtheorem{corollary}[corollary]{Corollary}
\aliascntresetthe{corollary}
\theoremstyle{definition}
\newaliascnt{definition}{theorem}

\aliascntresetthe{definition}
\newaliascnt{question}{theorem}

\aliascntresetthe{question}
\theoremstyle{remark}
\newaliascnt{remark}{theorem}
\newtheorem{remark}[remark]{Remark}
\aliascntresetthe{remark}
\crefname{proposition}{proposition}{propositions}
\Crefname{proposition}{Proposition}{Propositions}
\crefname{lemma}{lemma}{lemmas}
\Crefname{lemma}{Lemma}{Lemmas}
\crefname{corollary}{corollary}{corollaries}
\Crefname{corollary}{Corollary}{Corollaries}
\crefname{definition}{definition}{definitions}
\Crefname{definition}{Definition}{Definitions}
\crefname{question}{question}{questions}
\Crefname{question}{Question}{Questions}
\crefname{remark}{remark}{remarks}
\Crefname{remark}{Remark}{Remarks}
\newcommand{\R}{\mathbb R}
\newcommand{\C}{\mathbb C}
\newcommand{\Hh}{\mathbb H}
\newcommand{\Z}{\mathbb Z}
\newcommand{\Pcal}{\mathcal P}
\newcommand{\Scal}{\operatorname{Scal}}
\newcommand{\Hess}{\operatorname{Hess}}
\newcommand{\Ad}{\operatorname{Ad}}
\newcommand{\SO}{\operatorname{SO}}
\newcommand{\SU}{\operatorname{SU}}

\newcommand{\im}{\operatorname{Im}}
\newcommand{\supp}{\operatorname{supp}}
\newcommand{\op}{\mathrm{op}}
\newcommand{\II}{\mathrm{II}}

\newcommand{\Diff}{\operatorname{Diff}}
\newcommand{\Topg}{\operatorname{Top}}
\newcommand{\Map}{\operatorname{Map}}
\newcommand{\pr}{\operatorname{pr}}
\newcommand{\Kerv}{\Sigma_K^9}
\newcommand{\plumb}{\mathbin{\mathrm{plumb}}}
\numberwithin{equation}{section}
\title{\bfseries Positive sectional curvature on\\homotopy nine-spheres}
\author{Zhiqi Chen, Shaoqiang Deng, Junna Ni and Hui Zhang}
\date{}
\begin{document}
\maketitle
\vspace{-1.4em}
\begin{abstract}
We determine which homotopy nine-spheres admit positive sectional curvature: they are precisely the four smooth types with vanishing real Dirac index. The other four admit no metric of nonnegative sectional curvature. In particular, we construct a strictly positively curved metric on the exotic Kervaire nine-sphere, answering the longstanding nonnegative-curvature existence question in this dimension. The proof keeps explicit attaching maps throughout. A quaternionic commutator realizes the Kervaire sphere, a mixed representation realizes the Hopf image of the exotic eight-sphere, and a four-periodic normal-framing formula identifies the remaining index-zero class. Following the compatible-disk method, we give the curvature and boundary estimates for the first two models, including an explicit lower bound for the Kervaire metric in terms of the curvature and first covariant-derivative norms of a fixed connection, with the base scale and fiber normalization fixed. For the third model, we construct a positive metric on a cylinder joining a quotient metric to its pullback by a marked right Hopf twist. A four-form modification controls the connection terms involving both left and right quaternionic multiplication. Two positive disk fillings close the cylinder with the prescribed attaching map. A local convolution lemma supplies the smooth gluing while preserving a strict sectional-curvature lower bound.
\end{abstract}

\section{Introduction}\label{sec:intro}
We study positive sectional curvature on homotopy nine-spheres, with particular emphasis on the Kervaire nine-sphere. Our motivation comes from the relation between curvature and smooth structure. The standard sphere carries a metric of constant sectional curvature. Since Milnor's discovery of exotic spheres \cite{Milnor}, it has been natural to ask whether the same topological sphere, endowed with a different smooth structure, still admits a metric of positive sectional curvature.

This question has led to several important constructions. Gromoll and Meyer \cite{GM} found an exotic seven-sphere with nonnegative sectional curvature. Grove and Ziller \cite{GZ} constructed nonnegatively curved metrics on Milnor spheres, and Goette, Kerin and Shankar \cite{GKS} obtained such metrics on all exotic seven-spheres. Positive sectional curvature imposes a stronger condition. The differentiable sphere theorem of Brendle and Schoen \cite{BS}, for example, implies that a strictly pointwise quarter-pinched metric on a simply connected closed manifold has the standard spherical smooth type. Thus the construction of positive metrics on exotic spheres requires control of curvature without this pinching assumption.

The Kervaire spheres have played a central role in the higher-dimensional existence problem. In his 2003 paper, Totaro \cite[Section 5, p.~80]{Totaro} recorded that the conjecture asserting the existence of nonnegative sectional curvature on these spheres had already been open for about thirty years. The nine-dimensional exotic Kervaire sphere is a basic case of this problem. Our first objective is to settle this nine-dimensional case.

A major difficulty is that the natural symmetry does not provide such a metric. Grove, Verdiani, Wilking and Ziller \cite[Theorem A]{GVWZ} proved that the natural cohomogeneity one $\SO(2)\SO(5)$ action on the Brieskorn model of the Kervaire nine-sphere admits no invariant metric of nonnegative sectional curvature. They explicitly left open the existence of such a metric without the invariance requirement. This obstruction explains why a construction must allow more freedom than the natural cohomogeneity one ansatz. The disk construction used here provides that freedom and yields strictly positive sectional curvature.

In dimension nine the problem has a useful finite form. There are eight oriented smooth types of homotopy nine-spheres. The real Dirac index excludes positive scalar curvature on four of them \cite{Hitchin,CS}. On the other four, Stolz's theorem \cite{Stolz} gives positive scalar curvature. It is therefore natural to ask whether these same four smooth types admit positive sectional curvature. The scalar-curvature theorem alone does not answer this question, since sectional curvature must be positive on every tangent two-plane.

Recent constructions provide a way to approach the existence problem. He, Liu and Yau \cite{HLY} construct positive metrics on homotopy seven-spheres by decomposing each sphere into two disks and matching the metrics along the prescribed attaching diffeomorphism. Deng, Hu and Zhang \cite{DHZ} develop this method for a family of eleven-spheres and state a compatible-disk theorem for arbitrary finite-dimensional orthogonal $S^3$-representations. Galaz-Garc\'ia \cite{GG} combines the disk method with the models of Speran\c ca \cite{Speranca} to treat the exotic eight-sphere and the two order-three classes in dimension ten. We use these works in the preprint versions specified in the bibliography. The use of prescribed attaching maps also appears in the fifteen-dimensional construction of Chen, Deng and Zhang \cite{CDZ}. The first two models below fall within the compatible-disk framework of \cite[Proposition 5.1]{DHZ}. We give the needed dimension-nine constructions, including their curvature estimates and boundary comparison, in Sections~\ref{sec:bundle}--\ref{sec:gluing}.

Our first main result concerns the Kervaire sphere itself. Let $\Kerv$ denote the boundary of the plumbing of two copies of the tangent disk bundle of $S^5$.

\begin{theorem}[Positive curvature on the Kervaire nine-sphere]\label{thm:kervaire-existence}
The exotic Kervaire nine-sphere $\Kerv$ admits a smooth Riemannian metric of strictly positive sectional curvature.
\end{theorem}

The theorem gives an affirmative answer to the nonnegative-curvature existence question in dimension nine, with the stronger conclusion $\sec>0$. We prove this result by an explicit construction: we identify the smooth manifold in \cref{sec:clutching}, construct and glue its two disk metrics in Sections~\ref{sec:bundle}--\ref{sec:gluing}, and obtain a quantitative positive lower curvature bound in \cref{thm:positive}. This construction is the starting point for our study of the existence problem on all eight smooth nine-sphere types.

For $n\ge5$, let $\Theta_n$ denote the group of oriented homotopy $n$-spheres under connected sum, and set
\begin{equation}\label{eq:Pdef}
\Pcal_n=\{[\Sigma]\in\Theta_n:\Sigma\text{ admits a smooth metric with }\sec>0\}.
\end{equation}
Let $K\in\Theta_9$ be the Kervaire class, let $E\in\Theta_8$ be the exotic eight-sphere, and put $C_\eta=\eta E$, where $\eta$ acts by the Bredon pairing. We shall prove in \cref{prop:mixedclass} that $C_\eta\in\ker\alpha\setminus bP_{10}$. Together with the classical classification, this gives
\begin{equation}\label{eq:Theta9}
\Theta_9=\Z/2\{K\}\oplus\Z/2\{C_\eta\}\oplus\Z/2\{H\},
\end{equation}
where $H$ is any class with nonzero real Dirac index. The index values are
\begin{equation}\label{eq:alpha-values}
\alpha(K)=\alpha(C_\eta)=0,\qquad\alpha(H)=1.
\end{equation}
Our second main result completes the existence classification in dimension nine.

\begin{theorem}[Classification of the eight smooth types]\label{thm:classification}
A homotopy nine-sphere admits a smooth Riemannian metric of positive sectional curvature if and only if its real Dirac index vanishes. More precisely,
\begin{equation}\label{eq:classification}
\Pcal_9=\ker\alpha=\{0,K,C_\eta,K+C_\eta\}.
\end{equation}
The other four smooth types admit no smooth metric of nonnegative sectional curvature.
\end{theorem}

In particular, positive and nonnegative sectional curvature have the same existence criterion in dimension nine. The obstruction is classical. The main task is to construct positive metrics on the three nonstandard classes in $\ker\alpha$.

The proof has three geometric parts. The Kervaire metric is constructed in \cref{thm:positive}; the same disk method, with the representation norm retained in the estimates, gives the mixed metric in \cref{prop:mixedmetric}. The complementary class is treated by the marked right Hopf twist construction of \cref{thm:lastpositive}. Thus the curvature and boundary calculations needed for all three classes are included in the paper.

The disk architecture and principal connection estimates follow \cite{HLY,DHZ,GG}. Our additional ingredients are the identification of the three attaching maps, the four-periodic framing calculation, and the cylinder construction for the right Hopf marking. We call the specified boundary coordinates and attaching diffeomorphisms the \emph{marking}. Keeping them fixed connects the curvature construction to the smooth classification.

For $K$, a quaternionic Milnor--Munkres commutator identifies the double disk with the boundary of the plumbing of two copies of $D(TS^5)$. For $C_\eta$, a disk family for the mixed representation and the Morlet--Gromoll comparison identify the class in the full group $\Theta_9$. Changing the normal framing of a tube around a standard $S^3\subset S^8$ then gives
\[
t_m=(m\bmod2)K+\left(\binom m2\bmod2\right)C_\eta,
\]
which identifies the right Hopf tube twist with $K+C_\eta$.

The geometric difficulty for the last class is the simultaneous appearance of left and right multiplication on the three-sphere fiber. A four-form modification moves an order-one connection term from the mixed diagonal block to a horizontal--vertical cross block. The large positive vertical curvature absorbs this term, while concavity of the radius supplies the mixed curvature margin. The resulting cylinder and its two disk fillings construct a metric on the prescribed double directly. The curvature estimates and boundary comparisons are proved here, as is the local smoothing lemma in \cref{app:localgluing}; its conclusion agrees with \cite{RW}. The compatible-disk theorem of \cite{DHZ} and the smoothing theorem of \cite{RW} are therefore comparisons rather than existence inputs.

The paper is arranged as follows. In \cref{sec:obstruction} we recall the smooth classification and prove the index obstruction. Sections~\ref{sec:clutching}--\ref{sec:framing} construct and identify the three smooth models. Sections~\ref{sec:bundle}--\ref{sec:gluing} give compatible disk metrics for $K$ and $C_\eta$, with a quantitative lower bound in the Kervaire case. Sections~\ref{sec:bilateral}--\ref{sec:newcaps} prove the marked right Hopf twist theorem and complete the classification. The appendices contain the local smoothing proof, the coordinate derivation of the bilateral curvature blocks, the Kervaire parameter estimates, and reproduction instructions for the supplied algebraic checks.

\paragraph{Notation.}
We write $\Hh$ for the quaternions and identify $S^3$ with the unit quaternions. The symbols $L_q$, $R_q$ and $\Ad_q$ denote left multiplication, right multiplication and conjugation, respectively. All disks are closed unless otherwise indicated, and all disk diffeomorphisms used for gluing are smooth up to the boundary. The notation $\Diff_\partial(D^n)$ means diffeomorphisms equal to the identity on a boundary neighborhood. We assign opposite boundary orientation conventions to the two disks in $D^n\cup_fD^n$, so that an orientation-preserving coordinate map $f$ gives an oriented closed manifold. The stable Hopf element is denoted by $\eta$; scalar cutoff functions will be denoted by $\chi$ with subscripts.
For reference, the principal symbols used in the topology and the metric constructions are collected below. Radius profiles such as $F,r_N,r_S$ are defined separately for each construction.
\begin{center}
\small\renewcommand{\arraystretch}{1.12}
\begin{tabularx}{\textwidth}{@{}p{37mm}>{\raggedright\arraybackslash}X@{}}
\toprule
Symbol & Meaning \\
\midrule
$K,C_\eta,H$ & The three generators of $\Theta_9$ in \eqref{eq:Theta9}.\\
$\alpha,\eta,\Hh$ & The real Dirac index, the stable Hopf element, and the quaternions.\\
$\mathsf T(q),\mathcal V_z$ & The normal-framing map in \eqref{eq:framedtwists} and the infinitesimal action in \eqref{eq:Kmap}.\\
$\mathsf A,\Omega,M_0,M_1$ & A principal connection potential, its curvature, and the fixed norms of $\Omega$ and $D\Omega$.\\
$\tau_0,\varrho_\partial$ & The latitude of the Kervaire seam and its squared auxiliary-fiber radius.\\
$A_K,B_K,\varepsilon_K$ & The concavity constant, norm scale, and small parameter in \eqref{eq:parameters}.\\
$A_I,B_N,\varepsilon_I$ & The cylinder parameter, left-cap parameter, and common small parameter in Sections~\ref{sec:cylinder}--\ref{sec:newcaps}.\\
$R,\mathcal R,\xi^{\mathrm h}$ & The curvature tensor, curvature operator, and horizontal bivector component.\\
\bottomrule
\end{tabularx}
\end{center}

\section{Homotopy nine-spheres and the index obstruction}\label{sec:obstruction}
We first recall the topological information needed in the proof. It reduces the existence problem to three explicit constructions.

\subsection{The eight smooth types}
The group $\Theta_9$ has eight elements, and each element has order two \cite{KM,Sarkar}. Its subgroup $bP_{10}$ consists of homotopy spheres bounding parallelizable ten-manifolds and is generated by the nonstandard Kervaire sphere:
\[
bP_{10}=\{0,K\},\qquad
K=\left[\partial\bigl(D(TS^5)\plumb D(TS^5)\bigr)\right].
\]
The real Dirac index $\alpha:\Theta_9\to KO_9\cong\Z/2$ is surjective \cite{Hitchin,CS}. Hence its kernel has four elements. The identification of $C_\eta=\eta E$ in \cref{prop:mixedclass} will give the basis in \eqref{eq:Theta9}; in that basis,
\begin{equation}\label{eq:abc-index}
\alpha(aK+bC_\eta+cH)=c,\qquad a,b,c\in\{0,1\}.
\end{equation}
Curvature existence is unchanged by a diffeomorphism or by reversing orientation. Thus it is enough to construct one metric on a representative of each relevant class.

\subsection{The Dirac-index obstruction}
\begin{proposition}\label{prop:nonnegative-obstruction}
A homotopy nine-sphere with nonzero real Dirac index admits no smooth metric with $\sec\ge0$.
\end{proposition}
\begin{proof}
For a homotopy nine-sphere $\Sigma$, the vanishing of $H^2(\Sigma;\Z/2)$ gives $w_2(T\Sigma)=0$, so $\Sigma$ is spin. Its spin structure is unique since $H^1(\Sigma;\Z/2)=0$. The real index
\[
\alpha(\Sigma)\in KO_9\cong\Z/2
\]
is independent of the Riemannian metric. If it is nonzero, the Dirac operator for every metric has nonzero kernel \cite{Hitchin,CS}. In dimension $1\pmod8$, the index takes values in $\Z/2$.

Suppose $\sec_g\ge0$. Choose a nonzero harmonic spinor $\psi$. The Schr\"odinger--Lichnerowicz formula gives
\begin{equation}\label{eq:SL}
0=\int_\Sigma\left(|\nabla\psi|^2+\frac{\Scal_g}{4}|\psi|^2\right)\,dV_g.
\end{equation}
Since $\Scal_g\ge0$, both terms vanish. Hence $\psi$ is parallel; its nonzero norm is constant, and $\Scal_g\equiv0$. For any orthonormal frame,
\[
\Scal_g=2\sum_{i<j}\sec_g(e_i,e_j).
\]
A sum of nonnegative summands can vanish only if all summands vanish. Extending any chosen two-plane to an orthonormal frame therefore shows that every sectional curvature is zero. The metric is flat. A closed simply connected flat manifold would have to be isometric to its Euclidean universal cover, which is noncompact. This contradiction proves the proposition.
\end{proof}

\subsection{The remaining existence problem}
The proposition excludes the four classes with nonzero $H$-component. The standard sphere has its round metric. Thus the main theorem will follow once positive metrics have been constructed on $K$, $C_\eta$ and $K+C_\eta$. The following table records the models used for this purpose.
\begin{table}[htbp]
\centering\small
\renewcommand{\arraystretch}{1.18}
\begin{tabularx}{\textwidth}{@{}p{22mm}c >{\raggedright\arraybackslash}X >{\raggedright\arraybackslash}p{41mm}@{}}
\toprule
Class & $\alpha$ & Smooth model and construction & Result used \\
\midrule
$0$ & $0$ & Standard sphere; round metric. & Constant curvature $1$.\\
$K$ & $0$ & Quaternionic commutator; two quotient disks. & \Cref{thm:topology,thm:positive}; direct construction.\\
$C_\eta$ & $0$ & Mixed representation; compatible disks. & \Cref{prop:mixedclass,prop:mixedmetric,cor:mixedpositive}; direct construction.\\
$K+C_\eta$ & $0$ & Right Hopf tube twist; cylinder and two caps. & \Cref{cor:rightclass,thm:lastpositive}; direct construction.\\
$H+\ker\alpha$ & $1$ & The four remaining smooth types. & No $\sec\ge0$, by \cref{prop:nonnegative-obstruction}.\\
\bottomrule
\end{tabularx}
\caption{The eight smooth types, their models, and the dependence of each curvature conclusion.}\label{tab:types}
\end{table}

\section{A quaternionic model for the Kervaire sphere}\label{sec:clutching}
We begin with the class $K$. The plumbing description gives its smooth type, while the disk method requires an equivariant attaching map. The purpose of this section is to obtain both descriptions from the same quaternionic formula.

Consider the orthogonal representation
\begin{equation}\label{eq:representation}
V=\R\oplus\Hh^2,\qquad
\rho(q)(t,u,v)=(t,qu,qv),\quad q\in S^3.
\end{equation}
For a smooth map $\beta:S(V)\to S^3$ satisfying
\begin{equation}\label{eq:betaeq}
\beta(\rho(q)z)=q\beta(z)q^{-1},
\end{equation}
define $J_\beta(z)=\rho(\beta(z))z$. We shall construct such a map with the following property.

\begin{theorem}[Explicit smooth realization]\label{thm:topology}
There is a smooth map $\beta:S^8\to S^3$, given explicitly in \cref{sec:clutching}, which satisfies \eqref{eq:betaeq} for \eqref{eq:representation}. The map $J_\beta$ is an orientation-preserving diffeomorphism, and its double-disk manifold is diffeomorphic to
\begin{equation}\label{eq:Kdefinition}
D_N^9\cup_{J_\beta}D_S^9
\cong
\partial\bigl(D(TS^5)\plumb D(TS^5)\bigr)=\Kerv.
\end{equation}
The latter is the nonstandard Kervaire nine-sphere.
\end{theorem}

The construction starts with the suspended Hopf map. Its quaternionic equivariance makes the associated commutator easy to compute, and its stabilization is the clutching class of $TS^5$.

\subsection{A smooth suspended Hopf representative}
For $s\in\R$, define
\[
E_0(s)=\begin{cases}e^{-1/s},&s>0,\\0,&s\le0,\end{cases}
\qquad
\mathfrak h(s)=\frac{E_0(s)}{E_0(s)+E_0(1-s)}.
\]
The denominator is everywhere positive. Put
\begin{equation}\label{eq:cutoff}
\chi(s)=1-\mathfrak h\!\left(\frac{4s-1}{3}\right),
\qquad \vartheta(r)=\pi\chi(r^2).
\end{equation}
Then $\chi=1$ for $s\le1/4$ and $\chi=0$ for $s\ge1$. Fix a unit imaginary quaternion $i$. Define
\begin{equation}\label{eq:amap}
a(x)=
\begin{cases}
\displaystyle\cos\vartheta(|x|)+\sin\vartheta(|x|)\frac{x i\overline x}{|x|^2},&x\ne0,\\[4pt]
-1,&x=0.
\end{cases}
\end{equation}

\begin{lemma}\label{lem:amap}
The map $a:\Hh\to S^3$ is smooth, equals $-1$ on $|x|\le1/2$, and equals $1$ on $|x|\ge1$. It satisfies
\begin{equation}\label{eq:aeq}
a(qx)=q a(x)q^{-1},\qquad q\in S^3.
\end{equation}
Its one-point compactification $\bar a:S^4\to S^3$ represents the suspended Hopf element.
\end{lemma}
\begin{proof}
For $x\ne0$, the quaternion $xi\overline x/|x|^2$ is purely imaginary of norm one, so \eqref{eq:amap} has norm one. The apparent division by $|x|^2$ causes no singularity, since the full expression is the constant $-1$ on an entire neighborhood of the origin. It is constant $1$ outside the unit ball, so it also extends smoothly across the compactification point. Norm preservation and $(qx)i\overline{qx}=q(xi\overline x)q^{-1}$ prove \eqref{eq:aeq}.

Write $x=r\omega$ with $\omega\in S^3$. Then
\[
a(r\omega)=\cos\vartheta(r)+\sin\vartheta(r)\,\omega i\omega^{-1}.
\]
The angular map is the Hopf fibration $S^3\to S^2$, and the radial variable joins its two constant suspension ends. The constant end neighborhoods are a smooth reparameterization of the suspension. Reversing the suspension coordinate does not change the nonzero element of $\pi_4(S^3)\cong\Z/2$. This is the characteristic class of $\SU(3)\to S^5$ described in \cite[p.~5]{Puttmann}.
\end{proof}

\subsection{A compactly supported commutator on \texorpdfstring{$\Hh^2$}{the quaternionic plane}}
Set
\begin{equation}\label{eq:Cmap}
C(x,y)=a(x)a(y)a(x)^{-1}a(y)^{-1}.
\end{equation}
If either variable has norm at least one, one factor is $1$. If either variable has norm at most $1/2$, one factor is the central element $-1$. In both cases $C(x,y)=1$. Thus $C$ is the constant identity off a compact subset of $\Hh^2$; ``support'' below means the support of its difference from the identity.

For $z=(t,u,v)\in S(\R\oplus\Hh^2)$ and $t\ne1$, use stereographic coordinates
\begin{equation}\label{eq:stereo}
x=\frac{u}{1-t},\qquad y=\frac{v}{1-t}.
\end{equation}
The inverse has
\[
t=\frac{R^2-1}{R^2+1},\qquad
u=\frac{2x}{R^2+1},\qquad
v=\frac{2y}{R^2+1},\qquad R^2=|x|^2+|y|^2.
\]
Define
\begin{equation}\label{eq:explicitbeta}
\beta(t,u,v)=C\!\left(\frac{u}{1-t},\frac{v}{1-t}\right)\quad(t\ne1),
\qquad \beta(1,0,0)=1.
\end{equation}

\begin{lemma}\label{lem:explicitbeta}
The map \eqref{eq:explicitbeta} is smooth and satisfies \eqref{eq:betaeq}. It is identically $1$ whenever $|t|\ge1/3$. Its reentrance map $J_\beta$ is a diffeomorphism; in \eqref{eq:stereo} it is
\begin{equation}\label{eq:Psi}
\Psi(x,y)=(C(x,y)x,C(x,y)y).
\end{equation}
\end{lemma}
\begin{proof}
Smoothness at the north pole follows because $C=1$ near infinity. Equation \eqref{eq:aeq} gives $C(qx,qy)=qC(x,y)q^{-1}$, and stereographic projection commutes with the left action, proving equivariance. If $t\ge1/3$, then $R^2\ge2$ and at least one variable has norm at least one. If $t\le-1/3$, then $R^2\le1/2$ and at least one variable has norm at most $1/2$. This proves the asserted constant regions.

Equivariance with $q=\beta(z)$ gives $\beta(J_\beta(z))=\beta(z)$. Hence $J_\beta^{-1}=J_{\beta^{-1}}$. The map is the identity near a pole, so its local orientation sign, which is constant on the connected sphere, is positive. It fixes $t$ and left-multiplies both quaternionic coordinates; \eqref{eq:Psi} follows.
\end{proof}

\subsection{Exact identification with the Milnor--Munkres diffeomorphism}
Define
\begin{equation}\label{eq:FG}
F(x,y)=(x,a(x)y),\qquad G(x,y)=(a(y)x,y).
\end{equation}
Their inverses are obtained by replacing $a$ with $a^{-1}$ in the multiplied coordinate. Neither map is required to be compactly supported on the full product.

\begin{proposition}\label{prop:commutator}
With composition read from right to left,
\begin{equation}\label{eq:exactcommutator}
G^{-1}\circ F^{-1}\circ G\circ F=\Psi.
\end{equation}
\end{proposition}
\begin{proof}
Put $p=a(x)$, $q=a(y)$, $k=pqp^{-1}$, and $C=kq^{-1}$. Equivariance gives
\[
(G\circ F)(x,y)=(a(py)x,py)=(kx,py).
\]
Since $a(kx)=kpk^{-1}$,
\[
(F^{-1}\circ G\circ F)(x,y)
=(kx,kp^{-1}k^{-1}py)=(kx,Cy).
\]
Also $a(Cy)=CqC^{-1}$, so the final application of $G^{-1}$ gives
\[
((CqC^{-1})^{-1}kx,Cy)=(Cx,Cy),
\]
where $C^{-1}k=q$ was used. This is \eqref{eq:Psi}.
\end{proof}
Let $A(x)=L_{a(x)}\in\SO(4)$, where $L_p$ is quaternionic left multiplication. The maps in \eqref{eq:FG} are precisely the two fiber twists in the Milnor--Munkres pairing. The convention in \cite[Section 2.1]{CKS} is $G^{-1}F^{-1}GF$, matching \eqref{eq:exactcommutator}. This equality allows us to apply the plumbing correspondence to the attaching map just constructed.

\subsection{The stabilized clutch is the tangent bundle of \texorpdfstring{$S^5$}{the five-sphere}}
The principal bundle
\[
\SU(2)\longrightarrow\SU(3)\longrightarrow S^5
\]
is projection to the first column. Its characteristic map is the suspended Hopf map \cite{Puttmann}, hence agrees with $\bar a$ in $\pi_4(S^3)$. Under the standard representation of $\SU(2)$ on $\C^2\cong\Hh$, the associated real bundle satisfies
\begin{equation}\label{eq:tangent-association}
\SU(3)\times_{\SU(2)}(\R\oplus\C^2)\cong TS^5.
\end{equation}
Indeed, the map
\[
[B,(s,w)]\longmapsto (Be_1,B(is,w))
\]
is well defined and fiberwise linear. At $z=Be_1$,
\[
T_zS^5=\R\,iz\oplus z_\C^\perp,
\]
so it is a vector-bundle isomorphism. Consequently
\begin{equation}\label{eq:stabilized}
\operatorname{diag}(1,L_{\bar a}):S^4\longrightarrow\SO(5)
\end{equation}
is a clutching map for $TS^5$.

For $m=4$, the Milnor--Munkres/plumbing correspondence \cite[Lemma 2.1 and Corollary 2.2]{CKS} identifies the homotopy sphere defined by the commutator of two such twists with the boundary of the plumbing of their stabilized rank-five disk bundles. Equations \eqref{eq:exactcommutator} and \eqref{eq:stabilized} therefore give
\[
D_N^9\cup_{J_\beta}D_S^9
\cong\partial\bigl(D(TS^5)\plumb D(TS^5)\bigr).
\]
This boundary is nonstandard in dimension nine \cite{CH}. This proves \cref{thm:topology}.

\begin{remark}
The class $K$ has order two. Consequently either choice of overall orientation gives the same element of $\Theta_9$.
\end{remark}

\section{The mixed representation and the class \texorpdfstring{$C_\eta$}{C eta}}\label{sec:mixed}
We next construct a model for the second nonzero class in $\ker\alpha$. The starting point is Speran\c ca's model for the exotic eight-sphere \cite{Speranca}. We express its attaching map as a one-parameter family of diffeomorphisms of a fixed six-disk. Applying the Hopf operation to this family gives a two-parameter family and hence a homotopy nine-sphere. The main point is to identify the resulting class as $\eta E$ in $\Theta_9$.

The mixed representation is useful because, in a suitable chart, its attaching map is simply right multiplication in one $S^3$ coordinate. This will also make it possible to compare its normal framing with the Hopf models in the next section.

\subsection{An explicit equivariant map}
Write a point of $V_*=\R^2\oplus\Hh\oplus\im\Hh$ as $(s,t,h,w)$ and let
\begin{equation}\label{eq:mixedrep}
\rho_*(q)(s,t,h,w)=(s,t,qh,qwq^{-1}).
\end{equation}
Using the map $a$ of \eqref{eq:amap}, define
\begin{equation}\label{eq:cstar}
c_*(s,t,z)=a\bigl(4s+16(t+z)i(t-z)\bigr),\qquad z\in\im\Hh.
\end{equation}
The squared norm of the argument is $16s^2+256(t^2+|z|^2)^2$. Thus the support of $c_*-1$ lies in $s^2+t^2+|z|^2<1/8$. The proper map $(s,t,z)\mapsto4s+16(t+z)i(t-z)$ is a rescaled suspension of the Hopf map. Consequently its one-point compactification represents $\eta_4:S^5\to S^4$, and
\begin{equation}\label{eq:cstarclass}
[c_*]=\eta_3\eta_4\ne0\quad\hbox{in }\pi_5S^3=\Z/2.
\end{equation}
Here $\eta_j:S^{j+1}\to S^j$ denotes the suspended Hopf element. We shall relate this composition to the Bredon action in \cref{lem:hopfcomparison}.

For $h\ne0$, put $u=h/|h|$ and $z=u^{-1}wu$. On $S(V_*)$ set
\begin{equation}\label{eq:betastar}
\beta_*(s,t,h,w)=u c_*(s,t,z)u^{-1},
\end{equation}
and put $\beta_*=1$ when $h=0$. The support bound gives an entire neighborhood of $h=0$ on which the formula is constant, so this extension is smooth. It satisfies
$\beta_*(\rho_*(q)v)=q\beta_*(v)q^{-1}$.
In the free chart
\begin{equation}\label{eq:mixedchart}
\psi_2(u,p,z)=\bigl(p,\sqrt{1-|p|^2-|z|^2}\,u,uzu^{-1}\bigr),\qquad p=(s,t),
\end{equation}
the induced diffeomorphism has the particularly simple form
\begin{equation}\label{eq:mixedtwist}
J_{\beta_*}\psi_2(u,p,z)=\psi_2(uc_*(p,z),p,z).
\end{equation}
Its inverse replaces $c_*$ by $c_*^{-1}$ in this chart and is the identity elsewhere.

\subsection{The auxiliary eight-sphere and a fixed disk group}
Let
\[
b(y)=\cos\vartheta(|y|)+\sin\vartheta(|y|)\frac{y}{|y|},\qquad y\in\im\Hh\setminus\{0\},
\]
with $b(0)=-1$, using the same smooth cutoff as in \eqref{eq:cutoff}. It equals $1$ outside the unit ball and represents degree $\pm1$ after compactification. Define
\[
F_8(x,y)=(x,a(x)ya(x)^{-1}),\qquad G_8(x,y)=(b(y)x,y).
\]
As in \cref{prop:commutator}, equivariance gives
\begin{equation}\label{eq:ecommutator}
G_8^{-1}F_8^{-1}G_8F_8(x,y)=(C_8x,C_8yC_8^{-1}),\qquad C_8=[a(x),b(y)].
\end{equation}
The clutching classes are the generator $\Ad_a\in\pi_4\SO(3)$ and the stabilization of $L_b$ in $\pi_3\SO(5)$. The plumbing interpretation and the eight-dimensional example in \cite[Proposition 4, Corollary 5, and Section 3]{Speranca} identify its compactification with the nonstandard eight-sphere $E$. We use this identification to calibrate the disk family below.

Set $W=\Hh_L\oplus\im\Hh_{\Ad}$. In $S(W)=S^6$ use
\[
\psi(u,z)=\bigl(\sqrt{1-|z|^2}\,u,uzu^{-1}\bigr).
\]
Let $\mathcal G$ be the smooth pointwise multiplication group of maps $d:D^3\to S^3$ that are $1$ for $|z|\ge99/100$. It has the weak homotopy type of $\Map_*(S^3,S^3)$. Define $T_d\psi(u,z)=\psi(ud(z),z)$ and extend by the identity. All these supports lie in a common compact subset avoiding a point of the singular set $h=0$. After one stereographic projection and rescaling, they lie in the interior of one fixed six-disk. Hence
\begin{equation}\label{eq:groupmap}
\mathcal F:\mathcal G\longrightarrow\mathcal D:=\Diff_\partial(D^6),\qquad
\mathcal F(d)=T_{d^{-1}},
\end{equation}
is a continuous group homomorphism: $T_dT_e=T_{ed}$. Using one fixed disk allows us to apply the Gromoll construction to the whole family.

For $j=1,2$, write
\[
\psi_j(u,p,z)=\bigl(p,\sqrt{1-|p|^2-|z|^2}\,u,uzu^{-1}\bigr).
\]
For a function $c(p,z)$ supported away from the parameter boundary, the slice of its tube twist at $p$ is $\mathcal F(\gamma_c(p))$, where
\begin{equation}\label{eq:slicefamily}
\gamma_c(p)(z_0)=c\bigl(p,\sqrt{1-|p|^2}\,z_0\bigr)^{-1}.
\end{equation}
For the functions used here these are maps in $\mathcal G$. The family gives the sphere class $\Gamma_j\mathcal F_*[\gamma_c]$, with
$\Gamma_j:\pi_j\mathcal D\to\Theta_{j+7}$ the disk-family Gromoll map. The quotient map induced by $(p,z_0)\mapsto(p,\sqrt{1-|p|^2}z_0)$ has degree $+1$, so $[\gamma_c]=-[c]$. Both relevant homotopy groups have order two, making this sign immaterial. The family is constant near the parameter boundary; the square root therefore causes no smoothness problem there.

For the eight-dimensional commutator, the latitude $\lambda$ gives the slice $T_{d_\lambda}$, with
\begin{equation}\label{eq:Eslice}
\begin{split}
k(\lambda)&=\frac{\sqrt{1-\lambda^2}}{1-\lambda},\\
d_\lambda(z)&=\bigl[a(k(\lambda)\sqrt{1-|z|^2}),b(k(\lambda)z)\bigr].
\end{split}
\end{equation}
The first argument is a real quaternion. Nonidentity requires both arguments to have norms between $1/2$ and $1$; thus $|\lambda|<1/3$ and $|z|\le\sqrt{8/9}<99/100$. It follows that $g(\lambda)=d_\lambda^{-1}$ is a genuine based loop in $\mathcal G$, with
\begin{equation}\label{eq:Ecalibration}
\Gamma_1\mathcal F_*[g]=E.
\end{equation}
Since $E\ne0$ and $\pi_1\mathcal G=\pi_4S^3=\Z/2$, this loop represents $\eta_3$.

\subsection{The Hopf operation and the Gromoll map}
To identify the mixed sphere, we need the Hopf operation to commute with passage from a disk family to a homotopy sphere. We spell out this compatibility because it determines the $bP_{10}$ component as well as the image in $\operatorname{coker}J_9$.

For a topological group $G$, let $\partial_G:\pi_{j+1}BG\to\pi_jG$ be the boundary isomorphism of its universal bundle. For $x\in\pi_1G$, define
\begin{equation}\label{eq:Qoperation}
\mathcal Q_G(x)=\partial_G\bigl((\partial_G^{-1}x)\circ\eta_2\bigr).
\end{equation}
Naturality of the universal bundle shows that $\mathcal Q_G$ commutes with continuous group homomorphisms.

For the mapping group $\mathcal G$, a compatible classifying-space model is
\[
B\mathcal G\simeq\Map_*(S^3,BS^3)=\Omega^3BS^3.
\]
Indeed, $\Map_*(S^3,ES^3)$ is contractible, and every based map $S^3\to BS^3$ classifies a trivial bundle because $\pi_3BS^3=0$. Under threefold adjunction, precomposition by $\eta_2$ becomes precomposition by $\eta_5$. The universal-bundle boundary then gives
\begin{equation}\label{eq:QG}
\mathcal Q_{\mathcal G}(\eta_3)=\eta_3\eta_4.
\end{equation}

\begin{lemma}[Compatibility of the Hopf operation]\label{lem:hopfcomparison}
Let $\mathcal D=\Diff_\partial(D^6)$ and let $\Gamma_j:\pi_j\mathcal D\to\Theta_{j+7}$ be the disk-family Gromoll map. Then
\begin{equation}\label{eq:Qgromoll}
\Gamma_2\mathcal Q_{\mathcal D}(x)=\eta\,\Gamma_1(x)
\quad\text{in }\Theta_9,\qquad x\in\pi_1\mathcal D.
\end{equation}
\end{lemma}
\begin{proof}
Put $Y=\Topg(6)/O(6)$ and let $s:Y\to\Topg/O$ be stabilization. We use the delooped smoothing-section comparison, rather than choosing an arbitrary equivalence of the underlying loop spaces. For the disk of dimension six this gives
\begin{equation}\label{eq:deloopMorlet}
B\mathcal D\simeq\Omega^6_0Y;
\end{equation}
Here the subscript denotes the component of the constant map; fixing a boundary collar instead of fixing the boundary pointwise gives the same homotopy type. This is Morlet's equivalence in its classifying-space form \cite[Essay V, Theorem 3.4]{KS}, also recalled in \cite[p.~2, before Corollary C]{KK}; the smoothing-section construction and its stabilization compatibility are described in \cite[Theorems 1.2--1.3]{Burghelea}. We choose the equivalence induced by that construction. Composing the canonical equivalence $\mathcal D\simeq\Omega B\mathcal D$ with its loop gives the Morlet comparison used below. Thus the universal-bundle boundary and the delooping are part of the same comparison, not independently chosen identifications. Combining that boundary map, this equivalence, and sixfold adjunction gives isomorphisms
\[
\mu_j:\pi_j\mathcal D\longrightarrow\pi_{j+7}Y,
\qquad j=1,2.
\]
The definition of $\mathcal Q_{\mathcal D}$ now gives the commutative diagram
\[
\begin{array}{ccccc}
\pi_1\mathcal D&\xrightarrow{\ \mu_1\ }&\pi_8Y&\xrightarrow{\ s_*\ }&\pi_8(\Topg/O)\\[3pt]
\mathcal Q_{\mathcal D}\,\big\downarrow&&\big\downarrow\,{\circ\eta_8}&&\big\downarrow\,{\circ\eta_8}\\[3pt]
\pi_2\mathcal D&\xrightarrow{\ \mu_2\ }&\pi_9Y&\xrightarrow{\ s_*\ }&\pi_9(\Topg/O).
\end{array}
\]
For the left square, a representative $S^2\to B\mathcal D$ of $\partial_{\mathcal D}^{-1}x$ becomes, after the chosen equivalence and adjunction, a map $S^8\to Y$. Precomposing the former with $\eta_2$ precomposes the latter with $\Sigma^6\eta_2=\eta_8$. The right square is naturality of stabilization.

For precision, write $\mathfrak m:B\mathcal D\to\Omega^6_0Y$ for this chosen smoothing-section equivalence and $\operatorname{adj}_6$ for sixfold adjunction. Then the maps just used are
\begin{equation}\label{eq:comparison-composite}
\begin{split}
\mu_j&=\operatorname{adj}_6\circ\mathfrak m_*\circ\partial_{\mathcal D}^{-1},\\
\Gamma_j&=s_*\circ\operatorname{adj}_6\circ\mathfrak m_*\circ\partial_{\mathcal D}^{-1},
\qquad j=1,2.
\end{split}
\end{equation}
In the second line the target is identified with $\Theta_{j+7}$. In particular, the equivalence used to compute the Hopf operation and the one used to identify the disk-family sphere are the same smoothing comparison. The arrows and their compatibilities are summarized by
\[
\begin{array}{ccccc}
\pi_j\mathcal D&\xrightarrow{\ \partial_{\mathcal D}^{-1}\ }&\pi_{j+1}B\mathcal D
&\xrightarrow{\ \mathfrak m_*\ }&\pi_{j+1}\Omega^6_0Y\\
&&&&\big\downarrow\operatorname{adj}_6\\
\Theta_{j+7}&\xleftarrow{\ \cong\ }&\pi_{j+7}(\Topg/O)
&\xleftarrow{\ s_*\ }&\pi_{j+7}Y.
\end{array}
\]
The universal-bundle boundary supplies the first arrow, the smoothing-section construction supplies the second, and adjunction turns precomposition by $\eta_2$ into precomposition by $\Sigma^6\eta_2$. Stabilization commutes with that precomposition. The identification of the composite with $\Gamma_j$, rather than merely an abstract group isomorphism, is the stabilization compatibility discussed next.

Under $\Theta_{j+7}\cong\pi_{j+7}(\Topg/O)$, the composite $s_*\mu_j$ is $\Gamma_j$ for $j=1,2$. This is the stabilization compatibility of the smoothing-section and disk-family constructions; see \cite[Theorem 1.3]{Burghelea} and the Morlet--Gromoll stabilization diagram in \cite[Section 2.4]{Wang}. To compare conventions, the latter reference uses $\mathrm{PL}/O$. The stable smoothing fibration
\[
\mathrm{PL}/O\longrightarrow\Topg/O\longrightarrow\Topg/\mathrm{PL},
\qquad \Topg/\mathrm{PL}\simeq K(\Z/2,3),
\]
gives natural isomorphisms from $\pi_k(\mathrm{PL}/O)$ to $\pi_k(\Topg/O)$ for $k=8,9$ \cite{KS}. These isomorphisms commute with precomposition by $\eta_8$. The comparison with the cited diagram uses the natural maps from the unstable smoothing spaces to these stable spaces; no arbitrary identification of unstable $\mathrm{PL}$ and topological smoothing spaces is needed.

Finally, the compatibility with the Bredon pairing in \cite[Theorem 1.5]{Burghelea} identifies this precomposition operation with $\eta\,\Gamma_1(x)$. The final equality is obtained in the smoothing group $\pi_9(\Topg/O)\cong\Theta_9$ before any passage to $\operatorname{coker}J_9$. It therefore retains the $bP_{10}$ component. Reading the diagram proves \eqref{eq:Qgromoll}.
\end{proof}

\begin{proposition}[Identification of the mixed model]\label{prop:mixedclass}
The marked manifold $D^9\cup_{J_{\beta_*}}D^9$ represents $C_\eta=\eta E$. Moreover,
\[
C_\eta\in\ker\alpha\setminus bP_{10}.
\]
\end{proposition}
\begin{proof}
Apply \eqref{eq:slicefamily} to \eqref{eq:cstar}. Naturality of $\mathcal Q$, \eqref{eq:QG}, and \cref{lem:hopfcomparison} give
\begin{align*}
[D^9\cup_{J_{\beta_*}}D^9]
&=\Gamma_2\mathcal F_*(\eta_3\eta_4)\\
&=\Gamma_2\mathcal Q_{\mathcal D}(\mathcal F_*\eta_3)
=\eta\,\Gamma_1\mathcal F_*\eta_3=\eta E,
\end{align*}
using \eqref{eq:Ecalibration}. This identifies the actual marking in \eqref{eq:mixedtwist}.

For independence from $K$, we use the separate input \cite[Table 1 and Lemma 1.7(i)]{BKKT}. The table gives $\Theta_8^{\mathrm{split}}=0$, so $\eta E$ is not divisible by two in $\Theta_9$. The lemma says that this divisibility is equivalent after projection to $\operatorname{coker}J_9$. If $\eta E$ belonged to $bP_{10}$, that projection would be zero and hence divisible by two, a contradiction. The $d+1\equiv1\pmod8$ case in the proof of the same lemma places $\eta E$ in $b\mathrm{Spin}_{10}$. Spin-bordism invariance then gives $\alpha(\eta E)=0$.
\end{proof}

The formulas also show that the infinitesimal action is uniformly bounded. For a unit imaginary quaternion $\xi$,
\begin{equation}\label{eq:mixedactionbound}
|\xi h|^2+|\xi w-w\xi|^2\le |h|^2+4|w|^2\le4.
\end{equation}
Thus the action-norm bound $L_{\rho_*}\le2$ is available when we apply the disk theorem in \cref{sec:bundle}.

\section{Framed tube twists and the complementary class}\label{sec:framing}
The two preceding models can be viewed as twists supported near a standard three-sphere. Their difference is encoded by the normal framing of this three-sphere. We now compute how the resulting homotopy sphere changes when the framing changes. This will identify the attaching map used in the final metric construction.

\subsection{The family of marked twists}
Fix the standard core $S^3\subset S^8$ and choose $0<\rho_0<1$. We use the closed tubular neighborhood parametrized by
\[
\Psi_0(q,y)=\bigl(\sqrt{1-\rho_0^2|y|^2}\,q,\rho_0y\bigr),
\qquad(q,y)\in S^3\times D^5.
\]
The strict inequality $\rho_0<1$ makes this a smooth embedding up to the boundary. Changing $\rho_0$ by a positive isotopy does not change the attaching class. The factor $\rho_0$ also leaves the homotopy class of the normal framing unchanged.
When comparing this tube with the charts below, we restrict to smaller tubes and rescale the normal coordinate. Write $y=(z,w)\in\R\oplus\Hh$, and let $\mathsf T(q)=\operatorname{diag}(1,L_q)$. For an integer $m$ put
\begin{equation}\label{eq:framedtwists}
\begin{split}
\Psi_m(q,y)&=\Psi_0(q,\mathsf T(q)^m y),\qquad R_c(q,y)=(qc(y),y),\\
t_m&=[D^9\cup_{\Psi_mR_c\Psi_m^{-1}}D^9],
\end{split}
\end{equation}
where $c=1$ near $\partial D^5$ and $[c]=\eta_3\eta_4$. Each boundary map is extended by the identity outside the tube. A relative homotopy of $c$ gives an isotopy of the twist. A homotopy of normal framings gives an isotopy of tube embeddings. We may therefore shrink the support and change the smooth representative without changing $t_m$.

\begin{lemma}[Localization with fixed parameter collars]\label{lem:gauge-localization}
Let $P\to S^4$ be the left quaternionic Hopf bundle, let $\mathcal H$ be its smooth gauge group, and let $\mathcal H_*$ consist of gauge transformations equal to $1$ on a chosen fiber over $x_0$. Choose a closed four-ball $U\subset S^4$ with $x_0$ in the interior of its complementary ball. Gauge transformations supported in $\operatorname{int}U$ induce an isomorphism
\begin{equation}\label{eq:localgauge}
\pi_1\Map_*(S^4,S^3)\xrightarrow{\cong}\pi_1\mathcal H
\cong\pi_5S^3=\Z/2.
\end{equation}
A smooth based loop that is constant on a parameter collar can be homotoped to such a supported loop, through based loops with a common constant parameter collar. The corresponding gauge diffeomorphisms give an isotopy of the suspended attaching maps, relative to the two polar neighborhoods.
\end{lemma}
\begin{proof}
Evaluation on the chosen fiber is a locally trivial fibration of gauge groups. Since $\pi_2S^3=\pi_1S^3=0$, its homotopy exact sequence gives $\pi_1\mathcal H_*\cong\pi_1\mathcal H$.

Put $V=\overline{S^4\setminus U}$. A trivialization on $V$ identifies restriction of based gauge transformations with the map in the sequence
\[
\mathcal H_U\longrightarrow\mathcal H_*
\longrightarrow\Map_*((V,x_0),(S^3,1)).
\]
One may first work with continuous sections of the adjoint bundle. Restriction has the homotopy lifting property because the pair $(S^4,V)$ is a cofibration; the lifting property can also be obtained by extending a homotopy over a collar of $\partial V$. The target is contractible by contracting $V$ to $x_0$, and the fiber over the identity is the section space equal to $1$ on $V$. Trivializing on $U$ identifies this fiber with $\Map_*(U/\partial U,S^3)$. Relative smoothing of compact families and collar straightening give the same homotopy groups for smooth gauge transformations with constant collars. In particular, the inclusion of transformations supported in $\operatorname{int}U$ induces the claimed isomorphism on $\pi_1$.

The homotopies may be taken relative to the base point of the parameter circle. Precompose every loop and every intermediate loop with one fixed smooth reparameterization that is constant near that base point. Relative smoothing then makes the family smooth while retaining a common constant parameter collar. Such a family gives an isotopy of the suspended gauge diffeomorphisms: every slice is a gauge diffeomorphism, and constancy on the parameter collar makes every slice the identity on fixed polar neighborhoods. Finally, adjunction gives $\pi_1\Map_*(S^4,S^3)=\pi_5S^3$.
\end{proof}

\begin{lemma}[The first three framings]\label{lem:calibrations}
The classes in \eqref{eq:framedtwists} satisfy
\[
t_0=0,\qquad t_1=K,\qquad t_2=C_\eta.
\]
\end{lemma}
\begin{proof}
The twist with the product framing extends over $D^4\times D^5$ by $(v,y)\mapsto(vc(y),y)$. After rounding corners this is a disk extension, so $t_0=0$.

For $t_1$, use the commutator of \cref{prop:commutator}. On each latitude it is a gauge transformation of the left Hopf bundle $S^7\to S^4$, and its latitude loop is constant for $|\lambda|\ge1/3$. Apply \cref{lem:gauge-localization} to this loop. The latitude loop of the commutator cannot be null-homotopic, since that would isotope its sphere attaching map to the identity, contradicting \cref{thm:topology}. It therefore represents the unique nonzero local class in \eqref{eq:localgauge}. In a left Hopf tube the chart is
\[
(q,z,w)\longmapsto\bigl(\sqrt{1-z^2-|w|^2}\,q,z,qw\bigr).
\]
The core $w=z=0$ is the standard $S^3$. Differentiation in the normal coordinates gives $(z,w)\mapsto(z,qw)$, so its normal framing, relative to the product tube, is exactly $\mathsf T(q)$. A gauge transformation in this trivialization has the form $q\mapsto qc(z,w)$; the pair $(z,w)$ supplies the five normal parameters. By \cref{lem:gauge-localization}, the supported latitude loop represents the unique nonzero class in $\pi_5S^3$, with fixed parameter collars. Relative homotopy to the chosen representative $c$ therefore preserves the attaching isotopy class in this fixed tube. The local coordinates can be compared explicitly. On the latitude $\lambda=z$, a left Hopf section writes the two quaternionic coordinates as
\[
\bigl(\sqrt{1-z^2-|w|^2}\,q,qw\bigr).
\]
Thus $z$ is the suspension parameter and $w\in\Hh$ supplies the four local base coordinates. Shrinking the support into a smaller product neighborhood, and then applying a fixed rescaling, identifies $(z,w)$ with the normal disk used to define $R_c$. The loop homotopy from \cref{lem:gauge-localization} is constant on a common $z$-collar; compactly supported representatives and their relative homotopies can also be chosen constant on a common $w$-collar. Extending each slice by the identity therefore gives a smooth isotopy on $S^8$. The subsequent homotopy between representatives of the same element of $\pi_5S^3$ is likewise taken relative to a boundary collar of this normal disk. Consequently the class is $t_1=K$. The plumbing identification comes from \cref{thm:topology}; localization changes its representative while the displayed chart records its framing.

For $t_2$, the normal framing of \eqref{eq:mixedchart} along the core is $\operatorname{diag}(I_2,\Ad_q)$. In $\pi_3\SO(5)$ this is $2[L_q]$. To check the sign, let $Q(w)=\overline w$ and $P=\operatorname{diag}(-1,Q)\in\SO(5)$. Then
\begin{equation}\label{eq:rightframe}
P\operatorname{diag}(1,L_{q^{-1}})P^{-1}=\operatorname{diag}(1,R_q).
\end{equation}
Since $\SO(5)$ is connected, this implies $[R_q]=-[L_q]$. The identity $\Ad_q=L_qR_{q^{-1}}$ on $\Hh$ gives the asserted factor of two. The class of the tube twist in \eqref{eq:mixedtwist} is therefore $t_2$, which equals $C_\eta$ by \cref{prop:mixedclass}.
\end{proof}

The coordinate data in the calibrations, together with the right Hopf chart below, are collected in \cref{tab:framing-markings}. The normal coordinates are rescaled into smaller tubes before comparison; the core coordinate and the side of multiplication in the twist are kept fixed.
\begin{table}[htbp]
\centering\small\renewcommand{\arraystretch}{1.18}
\begin{tabularx}{\textwidth}{@{}p{24mm}p{31mm}>{\raggedright\arraybackslash}X p{21mm}@{}}
\toprule
Tube & Normal coordinates & Normal framing along the core & Twist \\
\midrule
Product & $y\in D^5$ & $I_5$ & $q\mapsto qc(y)$\\
Left Hopf & $(z,w)\in\R\oplus\Hh$ & $\operatorname{diag}(1,L_q)$ & $q\mapsto qc(z,w)$\\
Mixed & $(p,z)\in\R^2\oplus\im\Hh$ & $\operatorname{diag}(I_2,\Ad_q)$; class $2[L_q]$ & $q\mapsto qc_*(p,z)$\\
Right Hopf & $(z,x)\in\R\oplus\Hh$ & $\operatorname{diag}(1,R_q)$; class $-[L_q]$ & $q\mapsto qc(z,x)$\\
\bottomrule
\end{tabularx}
\caption{Core coordinate $q\in S^3$, normal framings, and the fixed side of multiplication. The right Hopf framing is computed in \cref{cor:rightclass}.}\label{tab:framing-markings}
\end{table}

\subsection{The relative mapping class calculation}
The three calibrations determine all the classes $t_m$ once we control the action of a framing change on the relative mapping class group. Krylov's theorem \cite[Lemma 1 and the preceding description of $FC_q^{p+1}$]{Krylov} makes this a finite-dimensional calculation over $\mathbb F_2$.

\begin{lemma}[Extension through the standard tube]\label{lem:tube-extension}
Extension by the identity from a standard tube $S^3\times D^5\subset S^8$, followed by the double-disk construction, induces a homomorphism
\[
e:\widetilde\pi_0\Diff(S^3\times D^5,\mathrm{rel}\,\partial)\longrightarrow\Theta_9.
\]
\end{lemma}
\begin{proof}
Use collar straightening to take all relative diffeomorphisms and pseudo-isotopies to be the identity on a boundary neighborhood. A relative pseudo-isotopy extends by the identity to a diffeomorphism of $S^8\times I$ with the two extended attaching maps as its ends. Glue two copies of $D^9\times I$ along their lateral boundaries by this map and round the corners. The result is a cobordism between the two double disks. Van Kampen's theorem makes it simply connected, and the relative Mayer--Vietoris sequence shows that each boundary inclusion is a homology equivalence. Hence both inclusions are homotopy equivalences. The smooth $h$-cobordism theorem gives an oriented diffeomorphism between the two double disks. Thus extension depends only on the relative pseudo-isotopy class. Composition of attaching maps corresponds to connected sum of the associated homotopy spheres, as in the standard description of $\Theta_9$ \cite{KM}; this proves the homomorphism assertion.
\end{proof}

\begin{proposition}[Four-periodicity]\label{prop:periodicity}
For every integer $m$,
\begin{equation}\label{eq:periodicity}
t_m=(m\bmod2)K+\left(\binom m2\bmod2\right)C_\eta.
\end{equation}
In particular, $t_{m+4}=t_m$ and $t_{-1}=t_3=K+C_\eta$.
\end{proposition}
\begin{proof}
Let $M=S^3\times D^5$ and let $\mathcal P=\widetilde\pi_0\Diff(M,\mathrm{rel}\,\partial)$, where the tilde denotes relative pseudo-isotopy classes. The cited theorem, with $(p,q)=(3,5)$, gives
\[
\mathcal P\cong\Theta_9\oplus FC_5^4,
\qquad FC_5^4\cong\pi_5\SO(4)=(\Z/2)^2;
\]
the second identification applies because $5<2\cdot3$. Hence $\mathcal P$ is a five-dimensional $\mathbb F_2$ vector space. Its disk-supported subgroup $L\cong\Theta_9$ has dimension three and $Q=\mathcal P/L$ has dimension two. We use the intrinsic quotient $Q$, so the subsequent calculation is independent of a choice of splitting.

Define
\[
e_0([f])=[y\mapsto\pr_1f(1,y)]\in\pi_5S^3=\mathbb F_2.
\]
A relative pseudo-isotopy, restricted to $\{1\}\times D^5\times I$ and projected to the first factor, gives a relative homotopy, so $e_0$ is well defined. It is also a homomorphism. Indeed, put $h_f(q,y)=\pr_1f(q,y)q^{-1}$. This descends to
$X=M/\partial M\simeq S^5\vee S^8$, and
$h_{fg}=(h_f\circ g)h_g$.
Every relative-boundary diffeomorphism fixes $H_3(M)$, and by oriented duality also $H_5(M,\partial M)$. Since $X$ is four-connected, the Hurewicz isomorphism gives $\bar g\circ j\simeq j$ for $j:S^5\to X$ induced by $y\mapsto(1,y)$. Restriction of the displayed product to $j$ proves additivity. A supporting disk can be moved off this slice, so $e_0|_L=0$, whereas $e_0([R_c])=[c]\ne0$.

Let $F(q,y)=(q,\mathsf T(q)y)$ and $\Phi([f])=[FfF^{-1}]$. Although $F$ need not fix the boundary, conjugation preserves the relative group. It fixes $L$ pointwise, since an orientation-preserving transport of a local disk preserves its sphere class. Also $F(1,y)=(1,y)$ and $\pr_1F=\pr_1$, whence $e_0\Phi=e_0$. Choose a basis of $Q$ in which the nonzero functional induced by $e_0$ is the first coordinate. Since $e_0\Phi=e_0$, the matrix of $\bar\Phi$ has the form $\left(\begin{smallmatrix}1&0\\b&1\end{smallmatrix}\right)$ with $b\in\mathbb F_2$. Consequently $(\bar\Phi-I)^2=0$. To display what this implies on the full relative group, choose any vector-space complement to $L$. In the resulting basis,
\begin{equation}\label{eq:framing-blocks}
\Phi=\begin{pmatrix}I_3&U\\0&J_b\end{pmatrix},\qquad
J_b=\begin{pmatrix}1&0\\b&1\end{pmatrix},\qquad
N_0^2=\begin{pmatrix}0&U(J_b-I_2)\\0&0\end{pmatrix},
\end{equation}
where $U$ is an arbitrary $3\times2$ matrix over $\mathbb F_2$ and $N_0=\Phi-I$. This auxiliary splitting is used only to display the calculation. Since $(J_b-I_2)^2=0$, it gives the intrinsic relations
\[
N_0^2\mathcal P\subset L,\qquad N_0L=0,\qquad N_0^3=0,\qquad\Phi^4=I.
\]
By \cref{lem:tube-extension}, extension through the standard tube gives a well-defined homomorphism $e:\mathcal P\to\Theta_9$. With $r_0=[R_c]$, one has $t_m=e(\Phi^m r_0)$. Lemma~\ref{lem:calibrations} gives
\[
e(r_0)=0,\qquad e(N_0r_0)=K,\qquad e(N_0^2r_0)=C_\eta.
\]
For example, $\Phi^2=I+N_0^2$ because the coefficient field has characteristic two, so the third equality follows from $t_2=C_\eta$. More generally, the truncated binomial expansion of $(I+N_0)^m$ proves \eqref{eq:periodicity} for nonnegative $m$. The identity $\Phi^4=I$ extends it to every integer: both $m\bmod2$ and $\binom m2\bmod2$ have period four, since $\binom{m+4}{2}-\binom m2=4m+6$.
\end{proof}

\subsection{The actual right Hopf marking}
On $S^8\subset\R\oplus\Hh^2$, write
\begin{equation}\label{eq:rightHopfchart}
v_1=\sqrt{\frac{1-z^2}{1+|x|^2}}\,q,\qquad
v_2=\sqrt{\frac{1-z^2}{1+|x|^2}}\,xq,\qquad q\in S^3.
\end{equation}
Choose a smooth representative $c(z,x)$ of $\eta_3\eta_4$ with $c-1$ supported in $|z|,|x|<5/8$, and define
\begin{equation}\label{eq:TR}
T_R(q,z,x)=(qc(z,x),z,x),
\end{equation}
extended by the identity. Such a representative is obtained from \eqref{eq:cstar} by an oriented linear identification $\R^2\oplus\im\Hh\cong\R\oplus\Hh$ and a rescaling of its compact support. The construction uses the smooth cutoff \eqref{eq:cutoff} throughout.

\begin{corollary}\label{cor:rightclass}
The marked double disk $D^9\cup_{T_R}D^9$ represents $K+C_\eta$.
\end{corollary}
\begin{proof}
The chart \eqref{eq:rightHopfchart} has normal framing $\operatorname{diag}(1,R_q)$. Its normal variable is related to the product tube by
$w=\sqrt{(1-z^2)/(1+|x|^2)}\,x$.
This is an orientation-preserving change of base coordinates on the support and is independent of $q$. By \eqref{eq:rightframe} the normal framing has class $-[L_q]$. The right multiplication in \eqref{eq:TR} is kept fixed during the framing homotopy. Thus its sphere class is $t_{-1}$, and \cref{prop:periodicity} applies.
\end{proof}

We have now identified all three attaching maps needed for the existence part of the main theorem. The remaining sections construct metrics that respect these maps.

\section{The compatible-disk construction}\label{sec:bundle}
We turn to the geometric part of the proof. The basic idea is to put invariant metrics on two product spaces $D^9\times S^3$ and then take a free quotient. Each quotient is a disk. Choosing the transition function and the boundary metrics together preserves the required attaching map. The following two nine-dimensional constructions are proved below directly from their source metrics. They are also special cases of \cite[Proposition 5.1]{DHZ}; that reference records the broader framework, rather than an existence theorem invoked in our proof.

\begin{theorem}[Compatible disks for the two nine-dimensional models]\label{thm:CD}
Let $(\rho,\beta)$ be the Kervaire model of \eqref{eq:representation} and \eqref{eq:explicitbeta}, or let $(\rho_*,\beta_*)$ be the mixed model of \eqref{eq:mixedrep} and \eqref{eq:betastar}. The two disks defining the corresponding marked double admit smooth metrics with positive sectional curvature. Their boundary metrics agree under the prescribed map $J_\beta$ or $J_{\beta_*}$, and the sum of their outward second fundamental forms is positive definite. The double therefore admits a smooth metric of positive sectional curvature.
\end{theorem}

The proof occupies Sections~\ref{sec:bundle}--\ref{sec:gluing}. We first give the construction for the Kervaire model, including explicit parameters, and then treat the mixed model in \cref{sec:mixedmetric}. The method and the principal connection estimate follow \cite{HLY,DHZ}; the quotient description is also used in \cite{GG}. The distinction between our two applications is the norm of the infinitesimal representation. It is at most one in the first model and at most two in the second. This norm enters the northern horizontal-plane estimate and is retained explicitly in the mixed case.

\subsection{Order of the choices}
The attaching maps and all support cutoffs are fixed before any radius parameter. The construction uses the following order; a later choice never changes the data that determine an earlier constant.
\begin{center}
\small\renewcommand{\arraystretch}{1.15}
\begin{tabularx}{\textwidth}{@{}p{27mm}>{\raggedright\arraybackslash}X@{}}
\toprule
Model & Order of choices \\
\midrule
Kervaire & $(\rho,\beta),\omega,h_B$; then $M_0,M_1$; then $A_K,\delta,\varepsilon_K$ by \eqref{eq:parameters}; then the smoothing width.\\
Mixed & $(\rho_*,\beta_*),\omega_*,h_B$; then $M_{0,*},M_{1,*},A_*$; then $\delta_*,\ell_*$; then $\varepsilon_*$ by \eqref{eq:mixedepsilon}; then the smoothing width.\\
Right Hopf & $c,\lambda,\mathcal U,h$; then $C$; then $A_I,B_N,\delta,\ell$; then $\varepsilon_I$ and $a=A_I\varepsilon_I^2$; finally the two smoothing widths.\\
\bottomrule
\end{tabularx}
\end{center}
For the right Hopf model the bounds defining $C$ are uniform on $0\le a\le1/50$, as established in \cref{prop:cylinderpositive}. For the Kervaire model the displayed curvature lower bound depends on the fixed connection norms and on the stated metric normalization; it is not a numerical evaluation of those norms.

\subsection{Two commuting actions}
The following bundle and quotient description applies to either pair in \cref{thm:CD}; we omit stars from the notation. Write $S^9=S(\R\oplus V)$ in polar coordinates $(\tau,z)$, where $z\in S(V)=S^8$ and $0\le\tau\le\pi$. Put $\gamma=\beta^{-1}$ and use the transition
\begin{equation}\label{eq:transition}
u_S=\gamma(z)u_N
\end{equation}
on the polar overlap. This defines a principal right $S^3$-bundle, with right action $u\mapsto uh$. Its commuting left star action is
\begin{equation}\label{eq:star}
q\star(\tau,z,u)=(\tau,\rho(q)z,qu).
\end{equation}
Equivariance gives $\gamma(\rho(q)z)qu_N=q\gamma(z)u_N$, so the action respects \eqref{eq:transition}. It is free because $qu=u$ forces $q=1$. Thus the quotient remains smooth even where the representation has fixed points. In extending disk coordinates $\zeta$, \eqref{eq:star} becomes $(\zeta,u)\mapsto(\rho(q)\zeta,qu)$ and is smooth at each pole.

Every star orbit in a product chart has a unique representative with fiber coordinate one. Hence its quotient coordinate is
\begin{equation}\label{eq:quotient-coordinate}
(\zeta,u)\longmapsto\rho(u^{-1})\zeta.
\end{equation}
Each quotient is a standard smooth nine-disk. At a common boundary point, set $y_N=\rho(u_N^{-1})z$ and $y_S=\rho(u_S^{-1})z$. From \eqref{eq:transition} and equivariance,
\begin{equation}\label{eq:marking}
y_S=\rho(\beta(y_N))y_N=J_\beta(y_N).
\end{equation}
Consequently, matching the metrics in the common source coordinates gives exactly the desired quotient identification.

\subsection{Dimensions and infinitesimal fields}
For \eqref{eq:representation}, the infinitesimal map is
\begin{equation}\label{eq:Kmap}
\mathcal V_z:\im\Hh\longrightarrow T_zS^8,
\qquad \mathcal V_{(t,u,v)}\xi=(0,\xi u,\xi v).
\end{equation}
Thus
\begin{equation}\label{eq:Kbound}
\mathcal V_z^*\mathcal V_z=(1-t^2)I_3,\qquad \|\mathcal V_z\|_\op\le1.
\end{equation}
The source disks have dimension twelve, while their star quotients have dimension nine. We will use two horizontal distributions. The principal projection has horizontal space $\ker\omega$, whereas the disk quotient uses the orthogonal complement of the star orbits. Curvature of the source is computed with the first splitting. O'Neill's formula for the quotient is applied to the second. On the northern disk it will be enough to prove positivity on these quotient-horizontal planes.

\section{Curvature of a principal connection metric}\label{sec:connection}
The southern disk will be obtained from a connection metric with a small three-sphere fiber. Small fibers control the horizontal curvature error, but mixed planes also require concavity of the radius function. We prove the needed estimate from the frame brackets and the Koszul formula in our normalization. This is the estimate developed in \cite[Equation (3.16)]{HLY} and \cite[Lemma 4.1]{DHZ}; the full component calculation is included here. The estimate reduces positivity on all two-planes to a quadratic form in three area components.

\subsection{Conventions and the estimate}
For the rest of the paper, use
\[
R(A,B)C=\nabla_A\nabla_BC-\nabla_B\nabla_AC-\nabla_{[A,B]}C,
\quad
\mathcal R(A\wedge B,C\wedge D)=-G(R(A,B)C,D).
\]
In particular $\mathcal R(A\wedge B,A\wedge B)=G(R(A,B)B,A)$. For a boundary with outward unit normal $n$, set $\II(X,Y)=G(\nabla_Xn,Y)$.

Let $\pi:P\to(B,g_B)$ be a principal right $S^3$-bundle, with fixed smooth connection $\omega$ and curvature $\Omega$. Normalize the inner product $Q$ on $\im\Hh$ by $[e_a,e_b]=2\varepsilon_{abc}e_c$, so the group has curvature one. For $r>0$, put
\begin{equation}\label{eq:connectionmetric}
G=\pi^*g_B+r^2Q(\omega,\omega),\qquad
\theta=d\log r,\qquad \mathsf N=-r^{-1}\Hess_Br.
\end{equation}
Set $M_0=\sup|\Omega|$ and $M_1=\sup|D\Omega|$, where
\[
|\Omega|^2=\sum_{i<j,a}(\Omega_{ij}^a)^2,
\qquad
|D\Omega|^2=\sum_{k,i<j,a}((D_k\Omega)_{ij}^a)^2.
\]
Their fully ordered component arrays therefore have norms $\sqrt2|\Omega|$ and $\sqrt2|D\Omega|$.

For a decomposable bivector in normalized horizontal and vertical frames, write
\begin{equation}\label{eq:decomp}
\xi=(X+U)\wedge(Y+V)=\xi^{\mathrm h}+\zeta+\nu,
\quad
\xi^{\mathrm h}=X\wedge Y,\quad
\zeta=X\otimes V-Y\otimes U,\quad
\nu=U\wedge V,
\end{equation}
and put $h=|\xi^{\mathrm h}|$, $m=|\zeta|$, $v=|\nu|$. Then $|\xi|^2=h^2+m^2+v^2$.

\begin{proposition}[Fine all-plane estimate]\label{prop:fine}
If $\sec_{g_B}\ge\kappa>0$, then
\begin{align}\label{eq:fine}
\mathcal R_G(\xi,\xi)\ge{}&
\left(\kappa-\tfrac34r^2M_0^2\right)h^2
+\left(\lambda_{\min}\mathsf N-\tfrac12r^2M_0^2\right)m^2
+(r^{-2}-|\theta|^2)v^2\notag\\
&-2r(M_1+2|\theta|M_0)hm
-(3M_0+r^2M_0^2)hv
-2r|\theta|M_0mv.
\end{align}
The constants do not increase with the base dimension.
\end{proposition}

\subsection{Frame derivatives and curvature blocks}
Choose a base-orthonormal frame normal at the point of calculation, and let $X_i$ be its horizontal lifts. Put $E_a=r^{-1}e_a^\#$ for the fundamental vector fields of the principal right action. At the point,
\begin{equation}\label{eq:framebrackets}
[X_i,X_j]^V=-r\Omega_{ij}^aE_a,
\quad [X_i,E_a]=-\theta_iE_a,
\quad [E_a,E_b]=r^{-1}c_{ab}^{\ c}E_c.
\end{equation}
The curvature components have a nonzero vertical derivative:
\begin{equation}\label{eq:Omega-vertical}
E_a(\Omega_{ij}^b)=-r^{-1}c_{ac}^{\ b}\Omega_{ij}^c.
\end{equation}
It follows from the adjoint-equivariance of the curvature form. This derivative contributes to the linear curvature terms below.

Koszul's formula gives
\begin{align}\label{eq:Koszul}
\nabla_{X_i}X_j&=(\nabla_i^BX_j)^H-\tfrac r2\Omega_{ij}^aE_a,
&\nabla_{X_i}E_a&=\tfrac r2\Omega_{ij}^aX_j,\notag\\
\nabla_{E_a}X_i&=\tfrac r2\Omega_{ij}^aX_j+\theta_iE_a,
&\nabla_{E_a}E_b&=\tfrac1{2r}c_{ab}^{\ c}E_c-\delta_{ab}\theta_iX_i.
\end{align}
Differentiating before evaluating at the normal point yields
\begin{align}
G(R(X_i,E_a)E_b,X_j)
&=\mathsf N_{ij}\delta_{ab}
+\tfrac{r^2}{4}\sum_k\Omega_{ik}^b\Omega_{jk}^a
-\tfrac14\langle[e_a,e_b],\Omega_{ij}\rangle,
\label{eq:mixed1}\\
G(R(X_i,X_j)E_a,E_b)
&=\tfrac{r^2}{4}\sum_k
(\Omega_{ik}^a\Omega_{jk}^b-\Omega_{jk}^a\Omega_{ik}^b)
+\tfrac12\langle[e_a,e_b],\Omega_{ij}\rangle,
\label{eq:mixed2}\\
G(R(X_i,E_a)E_b,E_c)
&=\tfrac r2\sum_k\theta_k(\delta_{ab}\Omega_{ik}^c-\delta_{ac}\Omega_{ik}^b),
\label{eq:mixed3}\\
G(R(X_i,X_j)X_k,E_a)
&=-\tfrac r2((D_i\Omega)_{jk}^a-(D_j\Omega)_{ik}^a)\notag\\
&\quad-\tfrac r2(\theta_i\Omega_{jk}^a-\theta_j\Omega_{ik}^a)
+r\theta_k\Omega_{ij}^a.
\label{eq:mixed4}
\end{align}
For clarity, the linear coefficient in \eqref{eq:mixed1} can be checked directly. The three terms in the curvature definition have horizontal components
\begin{align*}
G(\nabla_{X_i}\nabla_{E_a}E_b,X_j)
 &=\tfrac14c_{ab}^{\ c}\Omega_{ij}^c-\delta_{ab}\partial_i\theta_j,\\
G(\nabla_{E_a}\nabla_{X_i}E_b,X_j)
 &=-\tfrac12c_{ac}^{\ b}\Omega_{ij}^c
 +\tfrac{r^2}{4}\sum_k\Omega_{ik}^b\Omega_{kj}^a,\\
G(\nabla_{[X_i,E_a]}E_b,X_j)&=\delta_{ab}\theta_i\theta_j.
\end{align*}
Since $c_{ac}^{\ b}=-c_{ab}^{\ c}$, subtraction gives $-\frac14\langle[e_a,e_b],\Omega_{ij}\rangle$. If the first block is denoted $M_{ijab}$, Bianchi gives the second block as $M_{jiab}-M_{ijab}$. In \eqref{eq:mixed3}, the derivative of $c_{ab}^{\ c}/(2r)$ cancels its bracket counterpart. The last term in \eqref{eq:mixed4} comes from $-\nabla_{[X_i,X_j]}X_k$.

The pure horizontal numerator is
\[
\mathcal R_G(X\wedge Y,X\wedge Y)
=\mathcal R_B(X\wedge Y,X\wedge Y)-\tfrac34r^2|\Omega(X,Y)|^2
\]
by O'Neill's formula \cite{ONeill}. The vertical curvature has constant-curvature form with coefficient $r^{-2}-|\theta|^2$, from the Gauss equation for the round fibers.

\subsection{The unsuppressed linear term and decomposability}
The terms linear in $\Omega$ require special attention because they have no small factor $r$. Decomposability relates their mixed components to the horizontal and vertical areas. The components of \eqref{eq:decomp} satisfy the identity
\begin{equation}\label{eq:Pluecker}
\zeta_{ia}\zeta_{jb}-\zeta_{ib}\zeta_{ja}=\xi^{\mathrm h}_{ij}\nu_{ab}.
\end{equation}
This follows by expanding the four products. Put $J=\langle\Omega(X,Y),[U,V]\rangle$, identifying normalized vertical coefficient vectors with $\im\Hh$. Antisymmetry gives
\[
\sum_{i,j,a,b,c}c_{ab}^{\ c}\Omega_{ij}^c\zeta_{ia}\zeta_{jb}
=\tfrac12\sum_{i,j,a,b,c}c_{ab}^{\ c}\Omega_{ij}^c\xi^{\mathrm h}_{ij}\nu_{ab}=2J.
\]
The mixed--mixed contribution of \eqref{eq:mixed1} is therefore $-J/2$. In the horizontal--vertical cross block, the sign in $\mathcal R$ and the cross factor two give $-J$ from \eqref{eq:mixed2}. The exact sum is
\begin{equation}\label{eq:linearJ}
-\frac32\langle\Omega(X,Y),[U,V]\rangle.
\end{equation}
It need not vanish or be nonnegative. Since $|[U,V]|=2|U\wedge V|$,
\begin{equation}\label{eq:Jbound}
\left|\tfrac32\langle\Omega(X,Y),[U,V]\rangle\right|\le3M_0hv.
\end{equation}
The resulting loss is therefore controlled by the product $hv$, which can be absorbed by the horizontal and vertical curvature terms.

\subsection{The remaining norm factors}
For the quadratic part of \eqref{eq:mixed1}, the tensor
$T_{ijab}=\sum_k\Omega_{ik}^b\Omega_{jk}^a$ is an index rearrangement of a Gram matrix. Its full array norm is at most $2|\Omega|^2$. Its contraction therefore has absolute value at most $r^2M_0^2m^2/2$. The $\mathsf N$-term contributes at least $\lambda_{\min}(\mathsf N)m^2$.

The quadratic part of \eqref{eq:mixed2} has full ordered norm at most $r^2|\Omega|^2$. Restricting its two skew index pairs to $i<j$ and $a<b$ divides the norm by two; the quadratic-form cross factor is two. Its contribution is at most $r^2M_0^2hv$.

The full norm of \eqref{eq:mixed4} is bounded by
$\sqrt2rM_1+2\sqrt2r|\theta|M_0$. Restricting the first skew pair divides this by $\sqrt2$, then the cross factor two gives $2r(M_1+2|\theta|M_0)hm$.

Finally put $w_i^a=\sum_k\theta_k\Omega_{ik}^a$. Then
\[
\sum_{i,a}(w_i^a)^2\le2|\theta|^2|\Omega|^2,
\qquad
\sum_{a,b,c}(\delta_{ab}w_i^c-\delta_{ac}w_i^b)^2=4|w_i|^2.
\]
The factor four uses the vertical dimension three, not the base dimension. Including $r/2$, restricting to $b<c$, and including the cross factor gives $2r|\theta|M_0mv$. Together with the two pure blocks and \eqref{eq:Jbound}, this proves \cref{prop:fine}. The estimates remain valid when one or more of $h,m,v$ vanish.

\subsection{Small fibers with a fixed concavity margin}
\begin{corollary}\label{cor:smallfiber}
Suppose $B$ is compact, possibly with boundary, and a fixed function $\varphi$ satisfies
\begin{equation}\label{eq:concavity}
-\Hess\varphi\ge\Lambda g_B,
\qquad
\Lambda>M_0^2+\frac{2M_1^2}{\kappa}.
\end{equation}
For all sufficiently small $\varepsilon>0$, the metric with
$r^2=\varepsilon e^{\varepsilon\varphi}$ has positive sectional curvature.
\end{corollary}
\begin{proof}
Let $D_0=\sup|d\varphi|$ and $\mathcal{E}=e^{\varepsilon\varphi}$. Then
\[
\theta=\tfrac\varepsilon2d\varphi,
\qquad
\mathsf N=-\tfrac\varepsilon2\Hess\varphi
-\tfrac{\varepsilon^2}{4}d\varphi\otimes d\varphi.
\]
In \eqref{eq:fine}, use $H=h$, $M=\sqrt\varepsilon\,m$ and $V_*=v/\sqrt\varepsilon$. A lower-bound coefficient matrix has entries
\begin{align*}
Q_{11}&=\kappa-\tfrac34\varepsilon \mathcal{E} M_0^2,
&Q_{12}&=-\sqrt{\mathcal{E}}(M_1+\varepsilon D_0M_0),\\
Q_{22}&=\tfrac\Lambda2-\tfrac\varepsilon4D_0^2-\tfrac12\mathcal{E} M_0^2,
&Q_{13}&=-\tfrac{\sqrt\varepsilon}{2}(3M_0+\varepsilon \mathcal{E} M_0^2),\\
Q_{33}&=\mathcal{E}^{-1}-\tfrac{\varepsilon^3}{4}D_0^2,
&Q_{23}&=-\tfrac{\varepsilon^{3/2}}2\sqrt{\mathcal{E}} D_0M_0.
\end{align*}
The nonnegative variables $H,M,V_*$ allow negative off-diagonal entries to bound the possible losses. Uniformly on the compact base, the matrix tends to
\begin{equation}\label{eq:Qlimit-general}
\begin{pmatrix}
\kappa&-M_1&0\\
-M_1&(\Lambda-M_0^2)/2&0\\
0&0&1
\end{pmatrix}.
\end{equation}
Its upper-left determinant is $\kappa(\Lambda-M_0^2)/2-M_1^2>0$. Thus it is positive definite, and the matrices are uniformly positive for small $\varepsilon$. Since
$H^2+M^2+V_*^2\ge\varepsilon|\xi|^2$ for $0<\varepsilon\le1$, the assertion follows.
\end{proof}
The concavity hypothesis explains our use of a cap with boundary. It supplies the positive term on mixed planes; for comparison, a flat connection with constant radius has zero curvature on horizontal--vertical product planes.

\section{Positive metrics on the two quotient disks}\label{sec:disks}
We first choose the two source metrics for the Kervaire model. On the southern side, a round cap provides the concave function required by \cref{cor:smallfiber}. On the northern side, we choose a different radial profile so that the boundary radii match. This source metric need only be positive on the planes horizontal for the star quotient. The profiles and parameter strategy follow \cite[Sections 5 and 8]{DHZ}.

\subsection{Southern connection and fixed data}
Fix
\[
\tau_0=2\pi/3,\qquad F_0=\sin\tau_0=\sqrt3/2,
\qquad \gamma=\beta^{-1}.
\]
The southern base is the round cap $\tau_0\le\tau\le\pi$ with metric
\[
h_B=d\tau^2+\sin^2\tau\,h_{S^8}.
\]
Choose
\[
\chi_S(\tau)=\mathfrak h\!\left(\frac{3(\tau-\tau_0)}{\pi-\tau_0}-1\right).
\]
It is zero on a boundary collar and one near the south pole. Define the connection potentials in the common north gauge and the extending south gauge by
\begin{equation}\label{eq:potentials}
\mathsf A_N=\chi_S\gamma^{-1}d\gamma,
\qquad
\mathsf A_S=(\chi_S-1)d\gamma\,\gamma^{-1}.
\end{equation}
They obey
$\mathsf A_N=\gamma^{-1}\mathsf A_S\gamma+\gamma^{-1}d\gamma$, so they define one connection
$\omega=u^{-1}\mathsf Au+u^{-1}du$. Near the south pole $\mathsf A_S=0$; near the boundary $\mathsf A_N=0$. In particular, the extending gauge is smooth at the center, while the north gauge is convenient at the seam. In the common north gauge, writing $\vartheta_0=\gamma^{-1}d\gamma$ gives
\[
\Omega_N=d\chi_S\wedge\vartheta_0
+\chi_S(\chi_S-1)\vartheta_0\wedge\vartheta_0.
\]
Equivariance of $\gamma$ makes the connection star-invariant. Its smoothness on the compact cap gives finite $M_0=\sup|\Omega|$ and $M_1=\sup|D\Omega|$, computed with $h_B$. These constants are fixed before the fiber parameter is chosen.

Set
\begin{equation}\label{eq:parameters}
\begin{gathered}
B_K=1+M_0+M_1,\qquad
\Lambda=M_0^2+2M_1^2+2,\qquad
A_K=2\Lambda,\\
\lambda_0=(M_1^2+2)^{-1},\qquad
\delta=\frac1{16B_K},\qquad
\varepsilon_K=\frac1{2^{16}B_K^{12}}.
\end{gathered}
\end{equation}
The connection is fixed first. Thus $M_0$ and $M_1$ are fixed finite numbers, and the following choices of radii depend only on these data.

\subsection{The complete southern source}
Define
\begin{equation}\label{eq:GS}
G_S=h_B+r_S^2Q(\omega,\omega),\qquad
r_S(\tau)^2=\varepsilon_K e^{-\varepsilon_K A_K\cos\tau}.
\end{equation}
For $\varphi=-A_K\cos\tau$,
\[
-\Hess_{h_B}\varphi=-A_K\cos\tau\,h_B\ge\Lambda h_B.
\]
This satisfies \eqref{eq:concavity} with $\kappa=1$. More precisely, \cref{app:parameters} verifies for \eqref{eq:parameters} that
\begin{equation}\label{eq:southlower}
\sec(G_S)\ge\lambda_0\varepsilon_K/2>0.
\end{equation}
The star action is free and isometric. O'Neill's formula yields a metric $g_S$ on the entire quotient disk with the same lower bound. At the pole, the connection is flat in its extending gauge and $r_S(\pi-\rho)$ is an even smooth function of the polar distance $\rho$. This proves smoothness of the quotient metric at the pole.

The boundary data are
\begin{equation}\label{eq:southdata}
\varrho_\partial=r_a^2=\varepsilon_K e^{\varepsilon_K A_K/2},\qquad
q_0=\frac{r_S'(\tau_0)}{r_a}=\frac{\varepsilon_K A_K F_0}{2},\qquad
\mu_S=\cot\tau_0=-1/\sqrt3.
\end{equation}
In the common gauge, a normalized angular--fiber vector $X+U$ has southern outward second fundamental form
\begin{equation}\label{eq:BS}
B_S(X+U,X+U)=-\mu_S|X|^2-q_0|U|^2.
\end{equation}

\subsection{The northern radial profiles}
Let $\chi_N(v)=\mathfrak h(4v-1)$. Then $\chi_N=0$ for $v\le1/4$ and $\chi_N=1$ for $v\ge1/2$; a direct bound gives $M_\chi:=\sup|\chi_N'|\le32$. Define $F$ implicitly by
\begin{equation}\label{eq:Fimplicit}
s=F+\frac{F^3}{3\delta^2},\qquad
\ell=F_0+\frac{F_0^3}{3\delta^2},
\end{equation}
so $F(\ell)=F_0$. Put $d=r_aq_0$ and
\begin{equation}\label{eq:rN}
r_N(s)=r_a-d\int_s^\ell\chi_N(v/\delta)\,dv.
\end{equation}
On $D^9\times S^3$ take the flat-connection metric
\begin{equation}\label{eq:GN}
G_N=ds^2+F(s)^2h_{S^8}+r_N(s)^2h_{S^3}.
\end{equation}
A radial diffeomorphism preserving $z$ identifies this disk with the original north disk and preserves its boundary marking.

The explicit parameters satisfy
\begin{equation}\label{eq:deltaconditions}
\delta\le F_0/2,\qquad
\delta^2\le\frac1{16A_K F_0^2},\qquad
\delta^3\le\frac1{8A_K F_0(1+M_\chi)},
\end{equation}
and $q_0\ell<1/2$, $d<1$; see \cref{app:parameters}. Consequently
\begin{equation}\label{eq:rcontrol}
\frac{r_a}{2}\le r_N\le r_a,\quad
r_N'=d\chi_N(s/\delta)\ge0,\quad
r_N''=\frac d\delta\chi_N'(s/\delta)\ge0,\quad
\frac d{r_N^3}\le4A_K F_0.
\end{equation}
Near zero, $F$ is an odd smooth function with
$F(s)=s-s^3/(3\delta^2)+O(s^5)$, while $r_N$ is constant. Hence \eqref{eq:GN} extends smoothly over the center. At its endpoint,
\begin{equation}\label{eq:northdata}
F(\ell)=F_0,\quad r_N(\ell)=r_a,\quad
r_N'(\ell)/r_a=q_0,\quad
\mu_N=F'(\ell)/F_0>0.
\end{equation}

\subsection{Curvature on quotient-horizontal planes}
At fiber coordinate $u=1$, a star generator has normalized angular and vertical components $(F\mathcal V_z\xi,r_N\xi)$. Thus a star-horizontal vector has the form
\begin{equation}\label{eq:horizontalgraph}
\lambda\partial_s+X+U,\qquad U=-\frac F{r_N}\mathcal V_z^*X.
\end{equation}
Principal right translation transports this formula to every fiber coordinate. For two such vectors, \eqref{eq:Kbound} gives
\begin{align}
|\lambda V-\mu U|^2
&\le\frac{F^2}{r_N^2}|\lambda Y-\mu X|^2,
\label{eq:graph1}\\
|X\otimes V-Y\otimes U|^2
&\le\frac{2F^2}{r_N^2}|X\wedge Y|^2.
\label{eq:graph2}
\end{align}
For the second inequality, an area-preserving change of the pair $X,Y$ makes them orthogonal without changing either alternating expression. The resulting two tensor terms are orthogonal and each is controlled by the operator norm. If $X,Y$ are dependent, both sides vanish.

The full doubly warped sectional numerator is
\begin{align}\label{eq:warpedcurvature}
\mathcal K_N={}&-\frac{F''}{F}|\lambda Y-\mu X|^2
-\frac{r_N''}{r_N}|\lambda V-\mu U|^2\notag\\
&+\frac{1-F'^2}{F^2}|X\wedge Y|^2
+\frac{1-r_N'^2}{r_N^2}|U\wedge V|^2
-\frac{F'r_N'}{Fr_N}|X\otimes V-Y\otimes U|^2.
\end{align}
This is the complete five-block formula, as in \cite[Section 5.3]{DHZ}. Using \eqref{eq:graph1} and \eqref{eq:graph2}, it suffices to prove positivity of
\begin{equation}\label{eq:PradPang}
P_{\mathrm{rad}}=-\frac{F''}{F}-\frac{F^2r_N''}{r_N^3},\qquad
P_{\mathrm{ang}}=\frac{1-F'^2}{F^2}-\frac{2FF'r_N'}{r_N^3}.
\end{equation}
The remaining vertical coefficient is positive because $r_N'<1$.

Put $x=F/\delta$. Direct differentiation of \eqref{eq:Fimplicit} gives
\begin{equation}\label{eq:Fidentities}
-\frac{F''}{F}=\frac2{\delta^2(1+x^2)^3},\qquad
\frac{1-F'^2}{F^2}=\frac{2+x^2}{\delta^2(1+x^2)^2}.
\end{equation}
Moreover,
\[
\frac{(1-F'^2)/F^2}{FF'}
=\frac{2+x^2}{\delta^3x(1+x^2)}
\ge\frac1{\delta^2F_0}.
\]
Together with \eqref{eq:deltaconditions} and \eqref{eq:rcontrol}, this makes the angular loss at most half the positive angular term. On the support of $r_N''$, one has $s\le\delta/2$ and $F\le s$, so
\[
-\frac{F''}{F}\ge\frac{128}{125\delta^2},\qquad
\frac{F^2r_N''}{r_N^3}\le A_K F_0M_\chi\delta\le\frac1{8\delta^2}.
\]
Outside that support the radial loss is zero. Thus both coefficients in \eqref{eq:PradPang} are positive.

Projection of the horizontal graph to the radial--angular base tangent space is injective. For an independent pair its radial and angular base areas cannot both vanish. It follows that \eqref{eq:warpedcurvature} is positive on every star-horizontal two-plane. At the center the star orbit is the fiber, and the horizontal space is the base tangent space, whose sectional curvature tends to $2/\delta^2$. O'Neill's formula now gives a positively curved metric on the complete northern quotient disk.

For a quantitative bound, set
\[
T=F_0^2/\delta^2,\qquad
 a_*=[\delta^2(1+T)^3]^{-1},\qquad
H_*=1+4F_0^2/\varrho_\partial.
\]
The estimates above give $P_{\mathrm{rad}}\ge a_*$ and
$P_{\mathrm{ang}}\ge(2+T)/(2\delta^2(1+T)^2)\ge a_*$. The horizontal graph has squared operator norm at most $H_*$. Passing to second exterior powers therefore gives
\begin{equation}\label{eq:northlower}
\sec(g_N)\ge H_*^{-2}a_*
=\frac{\delta^4}{(\delta^2+3/4)^3}
\left(\frac{\varrho_\partial}{\varrho_\partial+3}\right)^2>0.
\end{equation}

\section{Boundary matching and the two positive metrics}\label{sec:gluing}
The source metrics constructed above give positive metrics on two quotient disks. To obtain a metric on $K$, we must compare their boundary values in the coordinates fixed in \cref{sec:bundle}. We first record the quotient formula and then check the second fundamental forms.

\subsection{The quotient metric}
The quotient metric can be written conveniently in terms of its inverse. In a local section $u=1$, write a source tangent vector as $(v,\xi)$ and the source norm as
\[
h(v,v)+r^2|\xi+\mathsf Av|^2.
\]
The quotient differential of \eqref{eq:quotient-coordinate} is
$P(v,\xi)=v-\mathcal V\xi$. With $\zeta=\xi+\mathsf Av$, this becomes
$P(v,\xi)=(I+\mathcal V\mathsf A)v-\mathcal V\zeta$. Consequently the quotient cometric is
\begin{equation}\label{eq:cometric}
g^{-1}=(I+\mathcal V\mathsf A)h^{-1}(I+\mathcal V\mathsf A)^*+r^{-2}\mathcal V\mathcal V^*.
\end{equation}
This is the quotient identity in \cite[Equation (6.2)]{DHZ}; the adjoints here mean the natural dual/transpose maps in fixed local bases.

On the northern disk $\mathsf A=0$. If $\mathcal V_a$ are the three fields associated with a $Q$-orthonormal basis, then
\[
g_N^{-1}=\partial_s\otimes\partial_s+F^{-2}h_{S^8}^{-1}
+r_N^{-2}\sum_{a=1}^3\mathcal V_a\otimes \mathcal V_a.
\]
Writing $\kappa_a=h_{S^8}(\mathcal V_a,\cdot)$ and using \eqref{eq:Kbound}, the covariant form is
\begin{equation}\label{eq:northcovariant}
g_N=ds^2+F^2h_{S^8}
-\frac{F^4}{r_N^2+F^2(1-t^2)}
\sum_{a=1}^3\kappa_a\otimes\kappa_a.
\end{equation}
At the center these polar formulas are interpreted using the smooth quotient construction already established. The southern quotient metric is given by \eqref{eq:cometric} with its extending-gauge potential \eqref{eq:potentials}, round cap metric, and radius \eqref{eq:GS}.

\subsection{Equality under the prescribed attaching map}
In the common north boundary gauge, both source boundary metrics are exactly
\begin{equation}\label{eq:sourceboundary}
F_0^2h_{S^8}+\varrho_\partial h_{S^3}.
\end{equation}
Their star actions agree. Thus the quotient boundaries are isometric under the common-gauge identification, whose expression in the two extending disk markings is \eqref{eq:marking}. Hence
\begin{equation}\label{eq:boundarymetric}
h_N=J_\beta^*h_S=:h_\partial.
\end{equation}
For a quotient boundary vector $Y$, choose its common horizontal lift $X+U$. The source outward normals $\partial_s$ and $-\partial_\tau$ are star-horizontal. Their second fundamental forms restrict to the corresponding quotient forms. By \eqref{eq:BS} and \eqref{eq:northdata},
\[
\II_N(Y,Y)=\mu_N|X|^2+q_0|U|^2,\qquad
(J_\beta^*\II_S)(Y,Y)=-\mu_S|X|^2-q_0|U|^2.
\]
The fiber terms cancel exactly:
\begin{equation}\label{eq:shapesum}
(\II_N+J_\beta^*\II_S)(Y,Y)=(\mu_N-\mu_S)|X|^2>0
\quad(Y\ne0).
\end{equation}
Strictness follows from the injective horizontal graph and $\mu_N>0>\mu_S$. Quantitatively,
\begin{equation}\label{eq:shapequant}
\II_N+J_\beta^*\II_S
\ge\frac{\varrho_\partial}{\varrho_\partial+F_0^2}(\mu_N-\mu_S)h_\partial>0.
\end{equation}
The isometry in \eqref{eq:boundarymetric} uses $J_\beta$, so the glued manifold is the Kervaire sphere identified in \cref{thm:topology}. The data used in this comparison are
\begin{center}
\small\renewcommand{\arraystretch}{1.15}
\begin{tabular}{@{}lcc@{}}
\toprule
Boundary datum & Northern disk & Southern disk, common gauge\\
\midrule
Base radius & $F_0$ & $F_0$\\
Auxiliary radius & $r_a$ & $r_a$\\
Outward normal & $\partial_s$ & $-\partial_\tau$\\
Base logarithmic derivative along the normal & $\mu_N$ & $-\mu_S$\\
Fiber logarithmic derivative along the normal & $q_0$ & $-q_0$\\
\bottomrule
\end{tabular}
\end{center}
The logarithmic derivatives are the coefficients of the second fundamental form on normalized source components $X$ and $U$. Both sides use the same quotient-horizontal lift.

\subsection{A local smoothing lemma}\label{sec:localsmoothing}
We record the smoothing statement used at all the seams in this paper. Its conclusion is the positive-sectional-curvature case of Reiser--Wraith \cite[Theorem A(i) and Corollary B]{RW}. We give a direct proof by convolution in a common normal coordinate in \cref{app:localgluing}, so that our metric constructions do not require that theorem as an input. The point to check is the quadratic error introduced by averaging the first normal derivative. The nonnegative derivative jump controls this error even when it has a kernel.

\begin{lemma}[Smoothing with a strict sectional-curvature threshold]\label{lem:localgluing}
Let $(M_-,g_-)$ and $(M_+,g_+)$ be compact smooth Riemannian manifolds of the same dimension, with distinguished boundary components identified by a smooth isometry $\phi$. Suppose that $\sec(g_\pm)>\kappa$ for a fixed real number $\kappa$, and that
\[
\II_-+\phi^*\II_+\ge0
\]
on the identified boundary, using outward normals. For every neighborhood of the seam in $M_-\cup_\phi M_+$, there is a smooth metric $g$ with $\sec(g)>\kappa$ which equals the original metrics outside that neighborhood. Other boundary components are unchanged.
\end{lemma}
The proof is given in \cref{app:localgluing}. It applies to semidefinite sums as well as definite sums; this distinction is used at the right seam in \cref{sec:newcaps}.

\subsection{Smoothing and a quantitative bound}
\begin{theorem}\label{thm:positive}
The Kervaire model of \cref{thm:topology}, with the disk metrics constructed above, admits a smooth metric $g_K$ satisfying
\begin{equation}\label{eq:mainbound}
\sec(g_K)>2^{-53}(1+M_0+M_1)^{-28}>0.
\end{equation}
Here $M_0$ and $M_1$ are the norms of the curvature and its first covariant derivative for the fixed southern connection in \cref{sec:disks}.
\end{theorem}
\begin{proof}
We apply \cref{lem:localgluing}, proved in \cref{app:localgluing}, with a strict lower curvature threshold shared by the two disks. The parameters of the disk metrics are fixed before choosing the smoothing width in that lemma.

Put
\[
\kappa_*=
\min\left\{\frac{\lambda_0\varepsilon_K}{2},\,
\frac{\delta^4}{(\delta^2+3/4)^3}
\left(\frac{\varrho_\partial}{\varrho_\partial+3}\right)^2\right\}>0.
\]
By \eqref{eq:parameters}, $\varrho_\partial\ge\varepsilon_K$, $\varepsilon_K<1/6$, and $\delta^2+3/4<1$. Thus \eqref{eq:northlower} gives
\[
\sec(g_N)>\frac{\varepsilon_K^2\delta^4}{16}.
\]
Also $\varepsilon_K<1/A_K$ implies $\varepsilon_K(M_1^2+2)<1/2$, hence
$\lambda_0\varepsilon_K/2>\varepsilon_K^2>\varepsilon_K^2\delta^4/16$.
Both disks therefore have a strict common lower threshold $\varepsilon_K^2\delta^4/32$. The boundary conditions \eqref{eq:boundarymetric} and \eqref{eq:shapequant}, together with \cref{lem:localgluing}, yield a smooth metric on $\Kerv$ satisfying
\[
\sec(g_K)>\frac{\varepsilon_K^2\delta^4}{32}
=\frac1{2^{53}(1+M_0+M_1)^{28}}.
\]
This proves the theorem.
\end{proof}
This proves \cref{thm:kervaire-existence} and gives the strictly positive metric underlying our answer to the Kervaire nonnegative-curvature question in dimension nine. The lower bound depends only on the chosen smooth connection and makes the small-parameter argument quantitative.

\begin{remark}[Meaning of the quantitative estimate]\label{rem:quantitative}
The base cap has curvature one, the inner product $Q$ is normalized in \cref{sec:connection}, and the connection and cutoffs are fixed in \cref{sec:disks}. Thus $M_0$ and $M_1$ in \eqref{eq:mainbound} are definite finite numbers for this construction. The estimate is explicit in these fixed geometric data and depends on this normalization. If certified bounds $\widehat M_i\ge M_i$ are available, one may use $\widehat M_0,\widehat M_1$ throughout \eqref{eq:parameters} and the subsequent estimates. The same construction then gives
\[
\sec(g_K)>2^{-53}(1+\widehat M_0+\widehat M_1)^{-28}.
\]
This explains how a numerical bound may be extracted without changing the proof.
\end{remark}

\subsection{The mixed model: the action-norm dependence}\label{sec:mixedmetric}
We now carry out the disk construction for $(\rho_*,\beta_*)$. This is the second case of \cref{thm:CD}. The principal connection estimate is unchanged; the northern horizontal-plane estimate must retain the norm of the representation. We give the choices and boundary comparison explicitly, following the compatible-disk method of \cite{DHZ}.

\begin{proposition}[Compatible metrics for the mixed model]\label{prop:mixedmetric}
The two disks marked by $J_{\beta_*}$ admit smooth metrics with positive sectional curvature, equal boundary metrics under $J_{\beta_*}$, and positive-definite sum of outward second fundamental forms.
\end{proposition}
\begin{proof}
Use the bundle and star action of \cref{sec:bundle} with $\rho,\beta,\gamma$ replaced by $\rho_*,\beta_*,\gamma_*:=\beta_*^{-1}$. The action is free on the source, and its quotient charts have transition $J_{\beta_*}$ by \eqref{eq:marking}. If $\mathcal V_{*,z}$ is its infinitesimal action on the unit sphere, then \eqref{eq:mixedactionbound} gives
\[
\|\mathcal V_{*,z}\|_{\op}\le L,\qquad L=2.
\]

\emph{Fixed connection and southern radius.}
Keep $\tau_0=2\pi/3$, $F_0=\sqrt3/2$, the round cap $h_B$, and the cutoff $\chi_S$ of \cref{sec:disks}. Define
\[
\mathsf A_{N,*}=\chi_S\gamma_*^{-1}d\gamma_*,\qquad
\mathsf A_{S,*}=(\chi_S-1)d\gamma_*\,\gamma_*^{-1}.
\]
These are potentials of one smooth star-invariant connection $\omega_*$. The extending south potential vanishes near the pole and the common north potential vanishes near the seam. Let $M_{0,*},M_{1,*}$ be its finite curvature and covariant-derivative norms with respect to $h_B$, using the normalization in \cref{sec:connection}, and set
\[
\Lambda_*=M_{0,*}^2+2M_{1,*}^2+2,\qquad A_*=2\Lambda_*.
\]
For a small parameter $\varepsilon_*>0$, put
\begin{equation}\label{eq:mixedsouth}
G_{S,*}=h_B+r_{S,*}^2Q(\omega_*,\omega_*),\qquad
r_{S,*}^2=\varepsilon_*e^{-\varepsilon_*A_*\cos\tau}.
\end{equation}
The function $\varphi_*=-A_*\cos\tau$ satisfies
$-\Hess\varphi_*\ge\Lambda_*h_B$ on the cap. By \cref{cor:smallfiber}, there is $\varepsilon_{S,*}>0$ such that $G_{S,*}$ has positive sectional curvature whenever $0<\varepsilon_*<\varepsilon_{S,*}$. Its free isometric quotient has the same positivity. At the pole the connection is flat in its extending gauge and $r_{S,*}$ is an even smooth function of the distance to the pole, so both source and quotient extend smoothly.

\emph{Northern parameters.}
Choose $\delta_*>0$, independently of $\varepsilon_*$, so that
\begin{equation}\label{eq:mixeddelta}
\delta_*\le F_0/2,\qquad
\delta_*^2\le\frac1{16L^2A_*F_0^2},\qquad
\delta_*^3\le\frac1{8L^2A_*F_0(1+M_\chi)}.
\end{equation}
Define $F_*(s)$ and $\ell_*$ by
\[
s=F_*+\frac{F_*^3}{3\delta_*^2},\qquad
\ell_*=F_0+\frac{F_0^3}{3\delta_*^2}.
\]
Here $F_*$ denotes this radial profile only in the present proof. The southern boundary radius and logarithmic slope are
\[
r_{a,*}^2=\varepsilon_*e^{\varepsilon_*A_*/2},\qquad
q_*=\frac{\varepsilon_*A_*F_0}{2},\qquad d_*=r_{a,*}q_*.
\]
After fixing $\delta_*$ and $\ell_*$, choose
\begin{equation}\label{eq:mixedepsilon}
0<\varepsilon_*<\min\left\{
\varepsilon_{S,*},\frac1{A_*},\frac1{A_*F_0\ell_*},
\left(\frac4{3A_*^2F_0^2}\right)^{1/3}\right\}.
\end{equation}
All entries are fixed positive numbers. These choices imply $q_*\ell_*<1/2$ and
$d_*^2<3\varepsilon_*^3A_*^2F_0^2/4<1$.
With the nondecreasing cutoff $\chi_N$ of \cref{sec:disks}, set
\begin{equation}\label{eq:mixednorth}
\begin{split}
r_{N,*}(s)&=r_{a,*}-d_*\int_s^{\ell_*}\chi_N(v/\delta_*)\,dv,\\
G_{N,*}&=ds^2+F_*^2h_{S^8}+r_{N,*}^2h_{S^3}.
\end{split}
\end{equation}
Then
\[
\frac{r_{a,*}}2<r_{N,*}\le r_{a,*},\qquad
0\le r_{N,*}'\le d_*<1,\qquad
\frac{d_*}{r_{N,*}^3}\le\frac{8q_*}{r_{a,*}^2}\le4A_*F_0.
\]

\emph{Northern curvature and smoothness.}
The quotient-horizontal graph is
$U=-(F_*/r_{N,*})\mathcal V_{*,z}^*X$.
The proof of \eqref{eq:graph1}--\eqref{eq:graph2}, now with operator norm bounded by $L$, gives
\begin{equation}\label{eq:mixedgraphs}
\begin{split}
|\lambda V-\mu U|^2&\le\frac{L^2F_*^2}{r_{N,*}^2}|\lambda Y-\mu X|^2,\\
|X\otimes V-Y\otimes U|^2&\le\frac{2L^2F_*^2}{r_{N,*}^2}|X\wedge Y|^2.
\end{split}
\end{equation}
Insert these bounds into \eqref{eq:warpedcurvature}. The two coefficients to control are
\[
P_{\mathrm{rad},*}=-\frac{F_*''}{F_*}-\frac{L^2F_*^2r_{N,*}''}{r_{N,*}^3},\qquad
P_{\mathrm{ang},*}=\frac{1-F_*'^2}{F_*^2}-\frac{2L^2F_*F_*'r_{N,*}'}{r_{N,*}^3}.
\]
The identities \eqref{eq:Fidentities} hold with $F,\delta$ replaced by $F_*,\delta_*$. In particular, for $s>0$,
\[
\frac{(1-F_*'^2)/F_*^2}{F_*F_*'}\ge\frac1{\delta_*^2F_0}.
\]
The ratio of the angular loss to the positive angular term is therefore at most
$8L^2A_*F_0^2\delta_*^2\le1/2$.
On the support of $r_{N,*}''$, one has $s\le\delta_*/2$ and $F_*\le s$, so
\[
-\frac{F_*''}{F_*}\ge\frac{128}{125\delta_*^2},\qquad
\frac{L^2F_*^2r_{N,*}''}{r_{N,*}^3}
\le L^2A_*F_0M_\chi\delta_*\le\frac1{8\delta_*^2}.
\]
Off this support the radial loss vanishes. Hence both coefficients are strictly positive; the remaining vertical coefficient is positive because $r_{N,*}'<1$. The projection of the horizontal graph to the disk tangent space is injective. Thus a nonzero horizontal two-plane has a nonzero radial or angular base area, and its curvature numerator is positive. O'Neill's formula gives positive sectional curvature on the punctured quotient disk.

Near its center, $F_*(s)=s-s^3/(3\delta_*^2)+O(s^5)$ is odd and smooth, and $r_{N,*}$ is constant. The source is therefore smooth. At the center the star orbit is the auxiliary fiber, and the quotient-horizontal planes have curvature $2/\delta_*^2$ in the disk factor. This proves positivity on the entire smooth quotient, including fixed points of the representation.

\emph{Boundary marking and second fundamental forms.}
At $s=\ell_*$, the northern values are
$F_*=F_0$, $r_{N,*}=r_{a,*}$, and $r_{N,*}'/r_{a,*}=q_*$.
In the common north gauge both source boundaries consequently have metric
$F_0^2h_{S^8}+r_{a,*}^2h_{S^3}$ and the same star action. Their quotient metrics agree under $J_{\beta_*}$. For the common horizontal lift $X+U$ of a boundary vector, the two outward forms sum to
\[
(\mu_{N,*}-\mu_S)|X|^2,\qquad
\mu_{N,*}=F_*'(\ell_*)/F_0>0,\quad\mu_S=-1/\sqrt3.
\]
The fiber terms $q_*|U|^2$ and $-q_*|U|^2$ cancel. The graph bound gives the quantitative inequality
\begin{equation}\label{eq:mixedshape}
\II_{N,*}+J_{\beta_*}^*\II_{S,*}
\ge\frac{r_{a,*}^2}{r_{a,*}^2+L^2F_0^2}
(\mu_{N,*}-\mu_S)h_{\partial,*}>0.
\end{equation}
For comparison with the Kervaire boundary table, the mixed data are
\begin{center}
\small\renewcommand{\arraystretch}{1.15}
\begin{tabular}{@{}lcc@{}}
\toprule
Boundary datum & Northern disk & Southern disk, common gauge\\
\midrule
Base and auxiliary radii & $(F_0,r_{a,*})$ & $(F_0,r_{a,*})$\\
Outward normal & $\partial_s$ & $-\partial_\tau$\\
Base logarithmic derivative along the normal & $\mu_{N,*}$ & $-\mu_S$\\
Fiber logarithmic derivative along the normal & $q_*$ & $-q_*$\\
\bottomrule
\end{tabular}
\end{center}
Here the boundary identification in the extending disk coordinates is $J_{\beta_*}$. This proves all the asserted boundary and curvature properties.
\end{proof}

\begin{corollary}\label{cor:mixedpositive}
The class $C_\eta$ admits a smooth Riemannian metric of positive sectional curvature.
\end{corollary}
\begin{proof}
Apply \cref{lem:localgluing} with $\kappa=0$ to the two metrics of \cref{prop:mixedmetric}. The gluing map is $J_{\beta_*}$, so \cref{prop:mixedclass} identifies the smooth class as $C_\eta$. Together with the Kervaire construction, this also completes the proof of \cref{thm:CD}.
\end{proof}

\section{Positive metrics for marked right Hopf twists}\label{sec:bilateral}
It remains to construct a positive metric on the class $K+C_\eta$. By \cref{cor:rightclass}, this amounts to using the right Hopf attaching map $T_R$. The geometric construction depends on the smooth map $c$ and its compact support, but not on its homotopy class. We state it in this form.

\begin{theorem}[Marked right Hopf twist construction]\label{thm:lastpositive}
Let $c:\R\oplus\Hh\to S^3$ be a fixed smooth map such that $c-1$ has compact support in $\{|z|<5/8,\ |x|<5/8\}$. In the right Hopf chart \eqref{eq:rightHopfchart}, put
\[
T_c(q,z,x)=(qc(z,x),z,x),
\]
and extend $T_c$ by the identity to $S^8$. There is a smooth positively curved cylinder $[0,1]\times S^8$ whose left collar is an ordinary Hopf quotient metric and whose right collar is the $T_c$ pullback of such a metric. The cylinder admits two positively curved disk fillings with matching boundary metrics and nonnegative sums of outward second fundamental forms. After smoothing, these pieces give a metric of positive sectional curvature on $D^9\cup_{T_c}D^9$. For the representative $[c]=\eta_3\eta_4$ chosen in \cref{sec:framing}, this sphere represents $K+C_\eta$.
\end{theorem}

The proof occupies this section and the next two. We write $T_R$ for the twist associated with the fixed map $c$ throughout the construction, and use its specified homotopy class only when identifying the final smooth type. All connection bounds may depend on $c$. We begin with the curvature calculation needed for the interpolation.

The interpolation naturally introduces two connection terms, acting on the round $S^3$ fiber by left and right multiplication. Both are isometries, and together they lie in $\SO(4)$. We call this a bilateral connection. Constant-radius formulas in this setting appear in \cite[Section 3.2]{BCG}. Here the radius must vary over the base in order to obtain a positive mixed-curvature term. We first compute its effect on the curvature tensor. We then use the four-form method of Thorpe, in the formulation of \cite{BM}, to put the estimate into a form suitable for the cylinder.

\subsection{Normal-gauge data}
Let $(B,h)$ be a Riemannian base, let $\sigma$ be the unit round metric on $S^3$, and let $u>0$ be a function on $B$. In a local trivialization write the horizontal lifts as $X_i=\partial_i-K_i$, where each $K_i$ is a fiber Killing field. The total metric is $h+u^2\sigma$ relative to this splitting. At a chosen base point use normal base coordinates and radial connection gauge, and at a chosen fiber point use normal fiber coordinates. Then, along the entire fiber over that base point,
\begin{equation}\label{eq:normalgauge}
K_i=0,\quad L_{ij}^a:=\partial_jK_i^a=-L_{ji}^a,\quad
M_{ijk}^a:=\partial_j\partial_kK_i^a,\quad S_{ijab}:=\partial_aL_{ij}^b.
\end{equation}
The curvature field satisfies $L_{ij}=-\Omega_{ij}^{\#}/2$. Since it is Killing,
\begin{equation}\label{eq:Sskew}
S_{ijab}=-S_{jiab}=-S_{ijba}.
\end{equation}
Write $u_i=\partial_i u$ and $u_{ij}=(\Hess_hu)_{ij}$. The orthonormal vertical frame at the point is $E_a=u^{-1}\partial_a$.

The coordinate metric and its two-jets are displayed in \cref{app:bilateral-derivation}. They give a direct computation of the following blocks, independent of a submersion curvature formula.

\subsection{The six curvature blocks}
We retain the curvature-operator convention of \cref{sec:connection} and explicitly set
\begin{equation}\label{eq:fourindex-convention}
R_{IJKL}:=\mathcal R(e_I\wedge e_J,e_K\wedge e_L)
=-g(R(e_I,e_J)e_K,e_L).
\end{equation}
Thus $R_{ijij}$ is positive on a unit positively curved coordinate two-plane. In particular, the $HHVV$ entry in this notation is the negative of $g(R(X_i,X_j)E_a,E_b)$ used in \eqref{eq:mixed2}.

\begin{proposition}\label{prop:bilateralblocks}
In the orthonormal frame $(X_i,E_a)$ the curvature components are
\begin{subequations}\label{eq:bilateralblocks}
\begin{align}
R_{ijkl}&=R^h_{ijkl}+u^2\sum_a
(-2L_{ij}^aL_{kl}^a-L_{ik}^aL_{jl}^a+L_{il}^aL_{jk}^a),\label{eq:biHHHH}\\
R_{ijka}&=u_jL_{ik}^a-u_iL_{jk}^a+2u_kL_{ij}^a
 +\frac u2(M_{ijk}^a-M_{jik}^a),\label{eq:biHHHV}\\
R_{ijab}&=2S_{ijab}+u^2\sum_k(L_{ik}^bL_{jk}^a-L_{ik}^aL_{jk}^b),\label{eq:biHHVV}\\
R_{iajb}&=-\frac{u_{ij}}u\delta_{ab}+S_{ijab}+u^2\sum_kL_{ik}^bL_{jk}^a,\label{eq:biHVHV}\\
R_{iabc}&=\sum_k u_k(L_{ik}^c\delta_{ab}-L_{ik}^b\delta_{ac}),\label{eq:biHVVV}\\
R_{abcd}&=\frac{1-|du|_h^2}{u^2}
(\delta_{ac}\delta_{bd}-\delta_{ad}\delta_{bc}).\label{eq:biVVVV}
\end{align}
\end{subequations}
Here the horizontal and vertical index ranges are independent; in the application they are six and three.
\end{proposition}
\begin{proof}
Insert the coordinate metric two-jets into \eqref{eq:coordinateR} and normalize each vertical index by $u^{-1}$. The complete derivation, with the second-derivative and Christoffel contributions separated for $HHVV$ and $HVHV$, is in \cref{app:bilateral-derivation}.
\end{proof}

\subsection{Removing the linear term from the mixed block}
The main difficulty in \eqref{eq:biHVHV} is the term $S_{ijab}$. It remains of order one when the fiber radius becomes small, whereas the Hessian margin in our construction will be small. We deal with it by modifying the curvature operator with a four-form. This changes the operator on general bivectors while preserving its value on every tangent two-plane.

A four-form on the total space defines a symmetric operator on its tangent bivectors by evaluation on pairs of bivectors. Let $H$ and $V$ denote the horizontal and vertical tangent spaces. The elementary identity
$\omega(v\wedge w,v\wedge w)=0$ is the basis of Thorpe's modified-curvature-operator method \cite[Sections 1--2]{BM}. It implies that positive definiteness of $\mathcal R+\omega$ is sufficient for positive sectional curvature. We use it on
\[
\Lambda^2(H\oplus V)=\Lambda^2H\oplus(H\wedge V)\oplus\Lambda^2V.
\]
Define $\omega$ to have only two-horizontal, two-vertical components, with
\begin{equation}\label{eq:chosenfourform}
\omega(X_i,X_j,E_a,E_b)=S_{ijab}.
\end{equation}
Equation~\eqref{eq:Sskew} makes this an alternating four-form. Since
$\omega(X_i,E_a,X_j,E_b)=-S_{ijab}$, the mixed diagonal block becomes
\begin{equation}\label{eq:modifiedHV}
(\mathcal R+\omega)|_{H\wedge V}=-u^{-1}\Hess_hu\otimes I_V+u^2\mathcal B,
\end{equation}
where the coefficients of $\mathcal B$ are quadratic in $L$. The $\Lambda^2H$--$\Lambda^2V$ block is $3S+u^2\mathcal C$. The order-one term now occurs in an off-diagonal block coupling horizontal and vertical areas. The vertical curvature is of order $u^{-2}$, and will absorb this coupling when $u$ is small.

On a fixed compact base set the connection and its derivatives are bounded, and the round fiber is compact. Radial gauges are obtained by parallel transport; their jets are controlled by the fixed connection and base metric. Hence $L,S,M$ in \eqref{eq:normalgauge} have uniform bounds over all base and fiber points. These bounds do not depend on the fiber-radius parameters to be chosen below. We will therefore be able to choose one radius parameter uniformly on the compact part of the cylinder where the interpolation takes place.

\section{A positive cylinder for the right Hopf twist}\label{sec:cylinder}
We now construct the interpolation. The base metric will have curvature operator bounded below by one. The fiber radius is chosen to be small and slightly concave; this gives the positive term on mixed planes. The attaching map $T_R$ and its representative $c$ are fixed throughout.

\subsection{The metric and its global extension}
We first fix the normalization of the ordinary quotient in the right Hopf chart \eqref{eq:rightHopfchart}. Write
\[
\phi=dq\,q^{-1},\qquad
\mathsf a=\frac{\im(\overline x\,dx)}{1+|x|^2},\qquad
h_5=\frac{dz^2}{1-z^2}
+\frac{(1-z^2)|dx|^2}{(1+|x|^2)^2}.
\]
Direct differentiation of \eqref{eq:rightHopfchart}, followed by completion of the square in $dq\,q^{-1}$, gives
\begin{equation}\label{eq:roundHopfsplit}
g_{S^8}=h_5+(1-z^2)|\phi+\mathsf a|^2.
\end{equation}
For example, if $\alpha=\overline x\,dx/(1+|x|^2)$, the real part of $\alpha$ is removed by differentiating the normalization $(1+|x|^2)^{-1/2}$. The cross term in the remaining quaternionic differential is $2\langle dq\,q^{-1},\im\alpha\rangle$. The residual $x$ term after completing the square is $|dx|^2/(1+|x|^2)^2$, which proves the displayed normalization.

Let $F,r>0$ depend only on the cylinder parameter $t$, and put $d=F^2(1-z^2)$. For the free diagonal quotient with the right action used in \cref{sec:newcaps}, identify the two fiber Lie algebras by the orthogonal adjoint map at $q$. A horizontal lift of a fixed fiber displacement of size $w\in\im\Hh$ then minimizes
\[
d|v|^2+r^2|\xi|^2\quad\text{subject to }v+\xi=w.
\]
Its minimum is $dr^2|w|^2/(d+r^2)$. Base directions horizontal for $\phi+\mathsf a$ are unchanged. Thus the ordinary quotient is
\begin{equation}\label{eq:ordinaryquotientderived}
dt^2+F^2h_5+u^2|\phi+\mathsf a|^2,
\qquad u^2=\frac{dr^2}{d+r^2},\quad u^{-2}=r^{-2}+d^{-1}.
\end{equation}
The fiber metric here is the unit round metric, with the quaternionic normalization fixed in \cref{sec:connection}. This derivation fixes both the action convention and the coefficient used below.

Put
\[
b=dc\,c^{-1},\qquad
\lambda(t)=\mathfrak h\bigl((t-1/10)/(4/5)\bigr).
\]
For $0\le t\le1$ and positive parameters $a,\varepsilon_I$, define
\begin{equation}\label{eq:cylinderprofiles}
\begin{split}
F(t)&=\cos(3/20+t),\qquad f_a(t)=1+a(t-t^2/4),\qquad r=\varepsilon_I f_a,\\
u^{-2}&=r^{-2}+[F^2(1-z^2)]^{-1},\qquad h=dt^2+F^2h_5,
\end{split}
\end{equation}
and
\begin{equation}\label{eq:bilateralmetric}
g=h+u^2|\phi+\mathsf a+\lambda\Ad_q b|^2.
\end{equation}
Here $a>0$ measures the concavity of the radius profile. The horizontal equation is
$dq=-\mathsf a q-q\lambda b$.
Both terms are Killing fields of the round fiber, so \cref{prop:bilateralblocks} applies with structure group $\SO(4)$.

Under a fiber coordinate change $q_j=kq_i$, the ordinary potential satisfies
$\mathsf a_j=\Ad_k\mathsf a_i-dk\,k^{-1}$.
The fiber form in \eqref{eq:bilateralmetric} transforms by $\Ad_k$, and its norm is unchanged. The additional potential $b$ is supported in the fixed free chart. Extending it by zero therefore gives a global connection there. Away from this support, \eqref{eq:bilateralmetric} is the ordinary right Hopf Cheeger quotient of
\begin{equation}\label{eq:ordinarysource}
dt^2+F^2g_{S^8}+r^2g_{S^3}.
\end{equation}
That free quotient defines the smooth metric at the poles $z=\pm1$ and in the other Hopf chart. It agrees with \eqref{eq:bilateralmetric} on the overlaps. Thus the quotient construction supplies the smooth extension across the singular Hopf coordinates, as in \cref{sec:gluing} and \cite{GG}.

Finally,
\begin{equation}\label{eq:TRpullback}
T_R^*(dq\,q^{-1}+\mathsf a)=\phi+\mathsf a+\Ad_qb.
\end{equation}
Since $\lambda$ is constant near each endpoint, the cylinder has the ordinary metric on its left collar and its $T_R$ pullback on its right collar. This identity holds for the whole collar metric, and hence for the boundary metric and its normal derivative.

\subsection{The compact-region estimate}
\begin{proposition}\label{prop:cylinderpositive}
For the fixed maps $c$ and $\lambda$, there are constants $A_{I,0}<\infty$ and $\varepsilon_{I,0}(A_I)>0$ such that \eqref{eq:bilateralmetric} has positive sectional curvature whenever
\[
A_I\ge A_{I,0},\qquad a=A_I\varepsilon_I^2,\qquad0<\varepsilon_I\le\varepsilon_{I,0}(A_I).
\]
The estimate holds on the entire smooth cylinder $[0,1]\times S^8$.
\end{proposition}
\begin{proof}
Choose a compact set $\mathcal U$ in the free base chart containing $\supp(dc)$ in its interior. Write $z=\sin\sigma$. Then
\[
h_5=d\sigma^2+\cos^2\sigma\,g_{S^4(1/2)}.
\]
Its curvature operator has eigenvalue $1$ on radial planes and $1+3/\cos^2\sigma$ on angular planes, with no cross blocks. Warping by $F$ in \eqref{eq:cylinderprofiles} gives $\mathcal R_h\ge I$. Also $-r^{-1}\Hess_hr$ has eigenvalues
\[
\frac{a}{2f_a},\qquad
\frac{a(1-t/2)\tan(3/20+t)}{f_a}\quad\hbox{with multiplicity five}.
\]
For $0<a\le1/50$ they are greater than $a/20$.

On $[0,1]\times\mathcal U$, the function $d=F^2(1-z^2)$ has a positive lower bound and
\[
v:=u/\varepsilon_I=f_a(1+\varepsilon_I^2f_a^2/d)^{-1/2}.
\]
This is a smooth positive function of $\tau=\varepsilon_I^2\in[0,1]$, $a\in[0,1/50]$, and the base point. Its derivatives on this compact set give a constant $C_0$ independent of $a,\varepsilon_I$ such that
\begin{equation}\label{eq:radiusbounds}
u\le2\varepsilon_I,\qquad |d\log u|\le C_0,\qquad
-u^{-1}\Hess_hu\ge(a/20-C_0\varepsilon_I^2)h.
\end{equation}
For the last inequality, the $\tau$ derivative of $-v^{-1}\Hess_hv$ is uniformly bounded, and its value at $\tau=0$ is $-f_a^{-1}\Hess_hf_a$. The normalization by $\varepsilon_I$ makes $v$ a smooth positive function even at $\tau=0$.

We make the constant dependence explicit before choosing $A_I$. On the fixed compact region $[0,1]\times\mathcal U$, let $K_L,K_S,K_M$ be uniform Euclidean array bounds for the normal-gauge tensors $L,S,M$ in \eqref{eq:normalgauge}, using orthonormal base and unit-fiber frames. Such bounds are fixed by $h$, the connection determined by $c,\lambda$, and a finite collection of bundle charts. They are independent of $a$ and $\varepsilon_I$, since neither the horizontal connection nor the base metric depends on these radius parameters.

Take $C_0$ as in \eqref{eq:radiusbounds}, uniformly for $0\le a\le1/50$ and $0\le\varepsilon_I^2\le1$. In the fixed dimensions six and three, each contraction in \eqref{eq:bilateralblocks} is bounded in operator norm by a numerical constant times the corresponding array norms. Hence a sufficiently large numerical constant $c_{6,3}$ makes
\begin{equation}\label{eq:cylinderconstant}
C=c_{6,3}(1+C_0)^2(1+K_L+K_S+K_M)^2\ge1
\end{equation}
a common bound for the coefficients below. Indeed, the horizontal and mixed quadratic errors are bounded by a multiple of $u^2K_L^2$; the three cross blocks are bounded by multiples of $u(K_M+C_0K_L)$, $K_S+u^2K_L^2$, and $uC_0K_L$, respectively. The vertical loss is $|d\log u|^2\le C_0^2$. This also explains why $C$ is fixed before $A_I$ and remains uniform when $a=A_I\varepsilon_I^2$ is imposed.

For a bivector $\xi$, denote the norms of its three components by $H_\xi,M_\xi,V_\xi$. Equations~\eqref{eq:bilateralblocks}--\eqref{eq:radiusbounds} now give
\begin{equation}\label{eq:threeblockbound}
\begin{split}
\langle(\mathcal R+\omega)\xi,\xi\rangle\ge{}&
(1-C\varepsilon_I^2)H_\xi^2+(a/20-C\varepsilon_I^2)M_\xi^2
+(1/(4\varepsilon_I^2)-C)V_\xi^2\\
&-2C\varepsilon_I H_\xi M_\xi-2CH_\xi V_\xi-2C\varepsilon_I M_\xi V_\xi.
\end{split}
\end{equation}
We now choose $a$ so that the mixed diagonal term dominates the losses from the cross terms. The first two cross terms are bounded respectively by
$H_\xi^2/4+4C^2\varepsilon_I^2M_\xi^2$ and $H_\xi^2/4+4C^2V_\xi^2$; the third is at most $\varepsilon_I^2M_\xi^2+C^2V_\xi^2$. Choose
\begin{equation}\label{eq:cylinderparameters}
\begin{gathered}
A_I\ge40(C+4C^2+1),\qquad a=A_I\varepsilon_I^2,\\
\varepsilon_I^2\le\min\left\{\frac1{50A_I},\frac1{4C},\frac1{8(C+5C^2)}\right\}.
\end{gathered}
\end{equation}
Then
\begin{equation}\label{eq:strictthreeblock}
\langle(\mathcal R+\omega)\xi,\xi\rangle
\ge\frac14H_\xi^2+\frac a{40}M_\xi^2+\frac1{8\varepsilon_I^2}V_\xi^2>0
\quad(\xi\ne0).
\end{equation}
The four-form vanishes on decomposable bivectors, giving strict sectional positivity throughout this compact region.

Outside $\mathcal U$, the source \eqref{eq:ordinarysource} has the five curvature-operator coefficients
\[
1,\quad1,\quad\frac{a}{2f_a},\quad\frac{1-r'^2}{r^2},\quad-\frac{F'r'}{Fr}.
\]
They are all positive after decreasing $\varepsilon_I$ so that $r'<1$. O'Neill's formula \cite{ONeill} gives positive curvature of the ordinary quotient, including neighborhoods of the collapsed Hopf orbits. Together with \eqref{eq:strictthreeblock}, this covers the full cylinder.
\end{proof}

\section{Disk fillings and the proof of the main theorem}\label{sec:newcaps}
The cylinder has positive curvature and the required boundary marking. We must now fill its two ends by disks. The right end can be filled by a quotient of a round cap with constant auxiliary radius. The left end requires a different profile: its boundary fiber slope must be slightly larger than that of the cylinder, so that the outward second fundamental forms have positive sum.

We use the horizontal-plane method of \cite[Section 3]{HLY} and \cite[Section 5.3]{DHZ}. The profile $F_N'=(1+F_N^2/\delta^2)^{-3/2}$ will give enough angular curvature throughout the left disk while allowing the prescribed boundary slope. We verify the curvature, smoothness at the center, and boundary comparisons in turn.

\subsection{A horizontal-plane criterion}
On $D^9\times S^3$ consider
\begin{equation}\label{eq:newcapsource}
G=ds^2+F_N(s)^2g_{S^8}+r_N(s)^2g_{S^3}.
\end{equation}
Let $\rho_R(h)(z,v_1,v_2)=(z,v_1h^{-1},v_2h^{-1})$ and let
$h\cdot(p,v)=(\rho_R(h)p,hv)$.
This is a free isometric action, with quotient coordinate $[p,v]\mapsto\rho_R(v)^{-1}p$. Thus the quotient is the standard smooth disk, with its usual boundary marking. A unit infinitesimal generator on the unit $S^8$ has norm at most one; we use the harmless bound $L=2$.

\begin{lemma}\label{lem:capcriterion}
Away from the center, the quotient of \eqref{eq:newcapsource} has positive sectional curvature if
\begin{equation}\label{eq:capcriterion}
\begin{gathered}
\frac{1-F_N'^2}{F_N^2}>
\frac{2L^2F_NF_N'r_N'}{r_N^3},\qquad
-\frac{F_N''}{F_N}>\frac{L^2F_N^2(r_N'')_+}{r_N^3},\\
0\le r_N'<1,\qquad F_N'\ge0.
\end{gathered}
\end{equation}
\end{lemma}
\begin{proof}
The horizontal space is a graph over the disk tangent space. For scaled angular and auxiliary-fiber components write $U=TX$, so $\|T\|\le LF_N/r_N$. For two horizontal vectors with radial coefficients $\lambda,\mu$, the graph inequalities \eqref{eq:graph1}--\eqref{eq:graph2} give
\[
|\lambda V-\mu U|^2\le\|T\|^2|\lambda Y-\mu X|^2,
\qquad |X\otimes V-Y\otimes U|^2\le2\|T\|^2|X\wedge Y|^2.
\]
Insert these into the doubly warped formula \eqref{eq:warpedcurvature}. The radial and angular coefficients are positive under \eqref{eq:capcriterion}, and the pure-fiber coefficient is nonnegative. The projection of a nonzero horizontal two-plane to the disk is two-dimensional, since the horizontal space is a graph. Hence its radial or angular base area is nonzero, making the source curvature numerator strictly positive. O'Neill's formula then gives the result. 
\end{proof}

\subsection{The left disk}
Fix $A_I$ as in \eqref{eq:cylinderparameters} and set
\[
F_* =\cos(3/20),\qquad B_N=11A_I/10.
\]
Use the nondecreasing cutoff $\chi_N(v)=\mathfrak h(4v-1)$ of \cref{sec:disks}, and write $M_\chi=\sup\chi_N'$. Choose
\begin{equation}\label{eq:newdelta}
0<\delta\le F_*/2,\qquad
\delta^3\le[32L^2B_N(1+M_\chi)]^{-1}.
\end{equation}
The first inequality in \eqref{eq:capcriterion} suggests choosing $F_N$ so that $(1-F_N'^2)/(F_N^3F_N')$ has a uniform positive lower bound. We use the following profile, for which this bound is $\delta^{-3}$, and define the cap length accordingly:
\begin{equation}\label{eq:newF}
F_N'=(1+F_N^2/\delta^2)^{-3/2},\quad F_N(0)=0,
\qquad \ell=\int_0^{F_*}(1+y^2/\delta^2)^{3/2}\,dy.
\end{equation}
Put $q_N=B_N\varepsilon_I^2$, $d_N=\varepsilon_I q_N$, and
\begin{equation}\label{eq:newr}
r_N(s)=\varepsilon_I-d_N\int_s^\ell\chi_N(v/\delta)\,dv.
\end{equation}
Further restrict $\varepsilon_I$ by $q_N\ell\le1/2$ and $d_N<1$. Then
\begin{equation}\label{eq:newrend}
\varepsilon_I/2\le r_N\le\varepsilon_I,\qquad
F_N(\ell)=F_*,\quad r_N(\ell)=\varepsilon_I,\quad
r_N'(\ell)/r_N(\ell)=q_N.
\end{equation}
In particular the boundary metric matches the ordinary left boundary of the cylinder.

\begin{proposition}\label{prop:leftcap}
The profiles \eqref{eq:newF}--\eqref{eq:newr} give a smooth positive-sectional-curvature metric on the whole quotient disk.
\end{proposition}
\begin{proof}
For $v=F_N^2/\delta^2>0$,
\[
\delta^3\frac{1-F_N'^2}{F_N^3F_N'}
=\frac{v^2+3v+3}{\sqrt v(1+v)^{3/2}}>1.
\]
Indeed the difference of the squared numerator and denominator is
$3v^3+12v^2+17v+9>0$.
Also $r_N'/r_N^3\le8B_N$. Equation~\eqref{eq:newdelta} therefore implies the first strict inequality in \eqref{eq:capcriterion} throughout $s>0$.

On the support of $r_N''$, one has $s\le\delta/2$ and $F_N\le s$. Direct differentiation of \eqref{eq:newF} gives
\[
-\frac{F_N''}{F_N}=\frac3{\delta^2}(1+F_N^2/\delta^2)^{-4}
\ge\frac{768}{625\delta^2},
\]
whereas
\[
\frac{L^2F_N^2r_N''}{r_N^3}
\le2L^2B_N M_\chi\delta
\le\frac{M_\chi}{16(1+M_\chi)\delta^2}.
\]
Off this support, $r_N''=0$ and $-F_N''/F_N>0$. The other conditions in \eqref{eq:capcriterion} follow from $0\le r_N'\le d_N<1$ and \eqref{eq:newF}.

It remains to include the center. The implicit function
$s=\int_0^{F_N}(1+y^2/\delta^2)^{3/2}\,dy$ is odd and analytic near zero, with derivative one. Its inverse $F_N(s)$ is therefore odd and analytic, with
$F_N(s)=s-s^3/(2\delta^2)+O(s^5)$.
This gives the full polar smoothness condition, and $r_N$ is constant near the center. The source disk is smooth there; its horizontal planes at the center lie in the disk factor and have limiting curvature $3/\delta^2>0$. The quotient is smooth and remains positively curved at the center. Lemma~\ref{lem:capcriterion} handles every other point.
\end{proof}

\subsection{The right cap and the two boundary forms}
Use outward second fundamental forms throughout, with the convention fixed in \cref{sec:connection,lem:localgluing}. At the left boundary, the source radii agree by \eqref{eq:newrend}, so their horizontal graph descriptions agree. For the scaled horizontal lift $X+U$ of a boundary tangent vector, the sum of the two outward forms is
\begin{equation}\label{eq:leftshapesum}
(\II_I+\II_N)(X+U,X+U)
=(\tan(3/20)+\mu_N)|X|^2+(q_N-a)|U|^2,
\end{equation}
where $\mu_N=F_N'(\ell)/F_*>0$ and $q_N-a=a/10$. Since $a\le1/50$, this is at least $(a/10)(|X|^2+|U|^2)$, including at the fixed points of the sphere action by smooth continuation.

For the right cap take the constant auxiliary radius and round-cap profile
\begin{equation}\label{eq:rightcap}
r_S=\varepsilon_I(1+3a/4),\qquad
F_S(\tau)=\cos(23/20+\tau),\quad0\le\tau\le\pi/2-23/20.
\end{equation}
Here $\tau=0$ is the boundary and the other endpoint is the center. Use the same free quotient action as above. The source is the product of a unit-curvature spherical cap with the round $S^3$ of radius $r_S$. Every quotient-horizontal two-plane projects injectively to the disk factor, so its source curvature numerator contains a strictly positive disk-area term; all remaining terms are nonnegative. O'Neill's formula proves positivity on the whole quotient. At the center, $F_S$ is the sine of the distance to the center, giving the usual smooth polar coordinates.

In ordinary boundary coordinates at the right endpoint of the cylinder, the horizontal coefficients in the two outward forms are $-\tan(23/20)$ and $+\tan(23/20)$. They cancel, leaving
\begin{equation}\label{eq:rightshapesum}
(\II_I+\II_S)(X+U,X+U)
=\frac{a}{2(1+3a/4)}|U|^2\ge0.
\end{equation}
The actual cylinder collar is the $T_R$ pullback of the ordinary collar by \eqref{eq:TRpullback}. Gluing the right cap by $T_R$ pulls back both the metric and \eqref{eq:rightshapesum} by the same map. Thus the boundary equality and the semidefinite form inequality hold for the prescribed marking. For clarity, all four boundary data are recorded below in ordinary coordinates; the two right-hand entries are then compared using $T_R$.
\begin{center}
\small\renewcommand{\arraystretch}{1.18}
\begin{tabular}{@{}lcccc@{}}
\toprule
Datum & Left disk & Cylinder left & Cylinder right & Right cap\\
\midrule
Base radius & $\cos(3/20)$ & $\cos(3/20)$ & $\cos(23/20)$ & $\cos(23/20)$\\
Fiber radius & $\varepsilon_I$ & $\varepsilon_I$ & $r_+$ & $r_+$\\
Outward normal & $\partial_s$ & $-\partial_t$ & $\partial_t$ & $-\partial_\tau$\\
Base slope & $\mu_N$ & $\tan(3/20)$ & $-\tan(23/20)$ & $\tan(23/20)$\\
Fiber slope & $q_N$ & $-a$ & $q_+$ & $0$\\
\bottomrule
\end{tabular}
\end{center}
Here $r_+=\varepsilon_I(1+3a/4)$ and $q_+=a/[2(1+3a/4)]$. A slope denotes the logarithmic derivative of the source radius along the outward normal. Restriction to the common horizontal graph gives the quotient second fundamental form, so the sums in \eqref{eq:leftshapesum} and \eqref{eq:rightshapesum} follow directly from the table.

\subsection{Smoothing and the common parameter choice}
\begin{proof}[Proof of \cref{thm:lastpositive}]
First fix $c,\lambda,\mathcal U$ and the finite constant $C$ in \eqref{eq:cylinderconstant}. Choose $A_I\ge40(C+4C^2+1)$, then $B_N,\delta,\ell$ as above. Finally choose $\varepsilon_I>0$ so that
\begin{equation}\label{eq:finalparameters}
\varepsilon_I^2\le\min\left\{
\frac1{50A_I},\frac1{4C},\frac1{8(C+5C^2)},\frac1{2B_N\ell}
\right\},
\end{equation}
and decrease it further to have $B_N\varepsilon_I^3<1$ and $r'<1$. All quantities on the right are fixed before $\varepsilon_I$ is chosen, so these requirements are compatible.

Propositions~\ref{prop:cylinderpositive} and \ref{prop:leftcap}, and the right-cap argument above, give positivity on the three complete pieces. Their boundaries are isometric under the stated identifications. The outward forms have positive-definite sum on the left and positive-semidefinite sum on the right by \eqref{eq:leftshapesum}--\eqref{eq:rightshapesum}. Apply \cref{lem:localgluing} with $\kappa=0$ at the two seams. Its normal-block estimate \eqref{eq:gluenormal} permits the semidefinite sum on the right, including its kernel directions. Choose disjoint seam neighborhoods and apply the lemma successively, so the first modification leaves the second seam unchanged. This gives a smooth positive metric. The cylinder and disk parameters are fixed first; the two smoothing widths are chosen last.

Topologically the cylinder is a standard boundary collar. The left cap is attached by the identity and the right cap by $T_R$, so the closed smooth manifold is $D^9\cup_{T_R}D^9$. This proves the assertion for every fixed $c$ in the statement. When $[c]=\eta_3\eta_4$ is the chosen representative, \cref{cor:rightclass} identifies the double as $K+C_\eta$.
\end{proof}

\begin{proof}[Proof of \cref{thm:classification}]
The standard sphere has positive sectional curvature. Theorem~\ref{thm:positive} gives such a metric on $K$, and \cref{cor:mixedpositive} gives one on $C_\eta$. Theorem~\ref{thm:lastpositive} gives a positive metric on $K+C_\eta$. These are all the elements of $\ker\alpha$ by \eqref{eq:Theta9}--\eqref{eq:alpha-values}. Conversely, \cref{prop:nonnegative-obstruction} excludes even nonnegative sectional curvature when $\alpha\ne0$. Hence \eqref{eq:classification} follows.
\end{proof}

\appendix
\section{Local smoothing with a strict curvature threshold}\label{app:localgluing}
This appendix proves \cref{lem:localgluing} in the normal collar coordinates of the prescribed boundary isometry. All estimates refer to fixed metrics; the smoothing width is the last parameter chosen.

\begin{proof}[Proof of \cref{lem:localgluing}]
\emph{Common collar and uniform estimates.}
Pull back the positive-side boundary coordinates by $\phi$ and use normal collars. On $(-2a,2a)\times X$, with $X$ the common compact boundary, the glued continuous metric is
\[
dt^2+h(t),\qquad h(0^-)=h(0^+)=h_0.
\]
The parameter increases from $M_-$ to $M_+$. The two families are smooth up to $t=0$ from their respective sides. Put
\begin{equation}\label{eq:gluejump}
A=h'(0^-),\quad B=h'(0^+),\quad D=A-B=2(\II_-+\phi^*\II_+)\ge0.
\end{equation}
All tensor comparisons below use $\bar g=dt^2+h_0$ and the fixed collar identification. On finitely many coordinate charts covering $X$, smoothness on the two closed half-collars gives uniform bounds for every derivative used below. In particular, for $|s|\le2\delta$, separately on each side,
\begin{equation}\label{eq:gluejets}
\begin{split}
\|h(s)-h_0\|_{C^2(X)}&\le C\delta,\\
\|h'(s)-A\|_{C^1(X)}&\le C\delta\quad(s<0),\\
\|h'(s)-B\|_{C^1(X)}&\le C\delta\quad(s>0),\qquad
\|h''(s)\|_{C^0(X)}\le C.
\end{split}
\end{equation}
Constants denoted by $C$ may increase, but depend only on the fixed metrics and collar, not on $\delta$. The notation $O(\delta)$ for a tensor means a bound $C\delta$ in its $\bar g$-norm, uniformly in the indicated collar. The width $a$ is fixed first and may be chosen inside any prescribed seam neighborhood.

\emph{The kernel and the distributional derivative.}
Choose the smooth even probability density
\[
\rho(v)=c_\rho\exp\bigl(-1/(1-v^2)\bigr)\quad(|v|<1),\qquad
\rho(v)=0\quad(|v|\ge1),
\]
where $c_\rho>0$ normalizes its integral. It is nonincreasing for $v\ge0$. For $0<\delta<a/4$, define, on $|t|<a$,
\[
\rho_\delta(t)=\delta^{-1}\rho(t/\delta),\qquad
\langle T\rangle_\delta(t)=\int_{-1}^1\rho(v)T(t-\delta v)\,dv,
\qquad h_\delta=\langle h\rangle_\delta.
\]
Averages use the fixed product identification of the tangent spaces; no metric is used to raise indices before averaging. Positivity of $h_\delta$ follows from positivity of the metrics being averaged. The family $h_\delta$ is smooth in both variables. Indeed, normal convolution smooths each piecewise smooth family, and all tangential derivatives commute with the integral.

Let $S=h'$ off the seam and let $h''_{\mathrm{reg}}$ denote the ordinary second derivative on the two sides. Continuity of $h$ and \eqref{eq:gluejump} give the distributional identities
\begin{equation}\label{eq:gluedistribution}
h_\delta'=\langle S\rangle_\delta,\qquad
h_\delta''=\langle h''_{\mathrm{reg}}\rangle_\delta-\rho_\delta(t)D.
\end{equation}
For $|t|\le\delta$, set
\[
\lambda(t)=\int_{t/\delta}^1\rho(v)\,dv,
\]
the mass sampling the negative side. The monotonicity and symmetry of $\rho$ imply
\begin{equation}\label{eq:kernelmass}
\lambda(1-\lambda)\le\delta\rho_\delta(t).
\end{equation}
For example, if $u=t/\delta\ge0$, then
$\lambda\le(1-u)\rho(u)\le\rho(u)$; if $u\le0$, the same argument gives $1-\lambda\le\rho(u)$. This also proves the inequality at the endpoints, where both sides vanish. Equations~\eqref{eq:gluejets} give, uniformly on this strip,
\begin{equation}\label{eq:glueaveragejets}
h_\delta=h_0+O(\delta)\ \text{in }C^2(X),\qquad
h_\delta'=\lambda A+(1-\lambda)B+O(\delta)\ \text{in }C^1(X).
\end{equation}

\emph{Curvature identities.}
For a smooth family $h$, put $S=h'$. In the curvature-operator convention of \cref{sec:connection}, the metric $dt^2+h$ satisfies
\begin{align}
\mathcal R(X\wedge Y,X\wedge Y)
&=\mathcal R_h(X\wedge Y,X\wedge Y)
-\tfrac14\bigl(S(X,X)S(Y,Y)-S(X,Y)^2\bigr),\label{eq:glueGauss}\\
\mathcal R(\partial_t\wedge Z,\partial_t\wedge W)
&=-\tfrac12h''(Z,W)+\tfrac14(Sh^{-1}S)(Z,W),\label{eq:glueRiccati}\\
\mathcal R(\partial_t\wedge X,Y\wedge Z)
&=\tfrac12\bigl((\nabla^h_ZS)(X,Y)-(\nabla^h_YS)(X,Z)\bigr).
\label{eq:glueCodazzi}
\end{align}
Here all capital-letter vectors are tangent to $X$. These formulas follow directly from
$\Gamma^t_{ij}=-S_{ij}/2$,
$\Gamma^i_{tj}=h^{ik}S_{kj}/2$, and
$\Gamma^i_{jk}=\Gamma^i_{jk}(h)$ in the coordinate curvature formula \eqref{eq:coordinateR}. Alternatively, they are the Gauss, normal-curvature, and Codazzi formulas for the slices. We compare their covariant components, so the possibly large term in \eqref{eq:gluedistribution} is never multiplied by an approximate inverse metric.

Let $\mathcal R_\delta$ be the curvature tensor of $dt^2+h_\delta$, and let $\overline{\mathcal R}_\delta=\langle\mathcal R\rangle_\delta$ average the ordinary curvature tensors of the original metrics off the seam. The value at the seam itself is immaterial to this integral.

\emph{The tangential block.}
For a symmetric bilinear form $T$ define
$\Delta_T(X,Y)=T(X,X)T(Y,Y)-T(X,Y)^2$.
For $\widehat S=\lambda A+(1-\lambda)B$, the exact quadratic identity is
\begin{equation}\label{eq:gluedetvariance}
\lambda\Delta_A+(1-\lambda)\Delta_B-\Delta_{\widehat S}
=\lambda(1-\lambda)\Delta_D.
\end{equation}
The intrinsic curvature tensor is a smooth function of a metric and its first two tangential derivatives on uniformly positive metrics. Thus \eqref{eq:gluejets}--\eqref{eq:glueaveragejets} imply
$\mathcal R_{h_\delta}-\langle\mathcal R_{h}\rangle_\delta=O(\delta)$.
Replacing $S(t-\delta v)$ by $A$ or $B$ in the quadratic terms also has error $O(\delta)$. It follows from \eqref{eq:glueGauss} and \eqref{eq:gluedetvariance} that
\begin{equation}\label{eq:gluetangent}
(\mathcal R_\delta-\overline{\mathcal R}_\delta)(X\wedge Y,X\wedge Y)
=\tfrac14\lambda(1-\lambda)\Delta_D(X,Y)
+O(\delta)|X\wedge Y|_{h_0}^2.
\end{equation}
The remainder is a uniformly bounded curvature-type tensor times $\delta$, which gives the stated area bound even for dependent or nearly dependent $X,Y$. Since $D\ge0$, its restriction to any two-plane has nonnegative determinant; hence $\Delta_D\ge0$.

\emph{The normal and cross blocks.}
Keeping the second-derivative average in \eqref{eq:gluedistribution} exact, use
$h_\delta^{-1}=h_0^{-1}+O(\delta)$ and the same estimate for every sampled $h^{-1}$. The quadratic identity
\[
\widehat S h_0^{-1}\widehat S
-\lambda A h_0^{-1}A-(1-\lambda)B h_0^{-1}B
=-\lambda(1-\lambda)D h_0^{-1}D
\]
and \eqref{eq:glueRiccati} yield
\begin{equation}\label{eq:gluenormal}
\begin{split}
(\mathcal R_\delta-\overline{\mathcal R}_\delta)
(\partial_t\wedge Z,\partial_t\wedge Z)
={}&\tfrac12\rho_\delta(t)D(Z,Z)\\
&-\tfrac14\lambda(1-\lambda)(Dh_0^{-1}D)(Z,Z)
+O(\delta)|Z|_{h_0}^2.
\end{split}
\end{equation}
Let $L=\sup_X\|h_0^{-1}D\|_{\op}<\infty$. The spectral theorem gives
$Dh_0^{-1}D\le LD$. By \eqref{eq:kernelmass}, the two displayed leading terms in \eqref{eq:gluenormal} have sum at least
\[
\rho_\delta(t)(1/2-\delta L/4)D(Z,Z)\ge0
\]
when $\delta L\le2$. This remains true in every kernel direction of $D$.

For the cross block, \eqref{eq:glueCodazzi} is linear in the first tangential derivatives of $S$. Tangential differentiation commutes with averaging. The Christoffel symbols of $h_\delta$ and of every sampled $h$ differ from those of $h_0$ by $O(\delta)$, while $S$ and its tangential derivatives are uniformly bounded. Therefore
\begin{equation}\label{eq:gluecross}
\bigl|(\mathcal R_\delta-\overline{\mathcal R}_\delta)
(\partial_t\wedge Z,\eta)\bigr|
\le C\delta|Z|_{h_0}|\eta|_{h_0},\qquad \eta\in\Lambda^2TX.
\end{equation}
In particular there is no unbounded cross-block error at the seam.

\emph{All sectional curvatures and localization.}
Every decomposable bivector is of the form
\[
\xi=(b\partial_t+X)\wedge(c\partial_t+Y)
=X\wedge Y+\partial_t\wedge(bY-cX).
\]
Its tangential component is decomposable. Combining \eqref{eq:gluetangent}--\eqref{eq:gluecross}, and using $2uv\le u^2+v^2$ for the cross error, gives
\begin{equation}\label{eq:glueallplanes}
\mathcal R_\delta(\xi,\xi)
\ge\overline{\mathcal R}_\delta(\xi,\xi)-C\delta|\xi|_{\bar g}^2
\qquad(|t|\le\delta).
\end{equation}
This argument uses positivity on decomposable bivectors only; no positive-curvature-operator hypothesis has been imposed on either original metric.

Compactness and the strict hypotheses give a number $\gamma>0$ with
$\sec(g_\pm)\ge\kappa+\gamma$. Applying this bound to the same decomposable $\xi$ in every sampled tangent space gives
\[
\overline{\mathcal R}_\delta(\xi,\xi)
\ge(\kappa+\gamma)\langle|\xi|_{dt^2+h}^2\rangle_\delta.
\]
On the crossing strip, the averaged squared norm differs from
$|\xi|_{dt^2+h_\delta}^2$ by at most $C\delta|\xi|_{\bar g}^2$, by \eqref{eq:gluejets}. These norms are uniformly equivalent. Equation~\eqref{eq:glueallplanes} consequently gives
$\sec(dt^2+h_\delta)\ge\kappa+\gamma/2$ there after decreasing $\delta$.
For $\delta\le|t|<a$, the convolution samples only one smooth side. The one-sided smooth derivative bounds give uniform $C^2$ convergence to the original metric on this region, and hence $\sec>\kappa$ there as well.

Finally choose a fixed smooth cutoff $\chi(t)$ equal to one for $|t|\le a/2$ and zero for $|t|\ge3a/4$, and replace $h$ by
\[
\widetilde h_\delta=h+\chi(t)(h_\delta-h).
\]
This is smooth across zero, positive definite, and equals $h$ outside the chosen collar. On the support of $d\chi$ the original family is smooth and $h_\delta-h\to0$ in $C^2$; the cutoff and its derivatives are fixed before $\delta$ is chosen. The same strict curvature margin therefore survives this final modification. This proves the lemma with the prescribed attaching map and support.
\end{proof}

\section{Coordinate derivation of the bilateral curvature blocks}\label{app:bilateral-derivation}
This appendix proves \cref{prop:bilateralblocks} directly from the coordinate metric. The base metric is $h$, the fiber metric $\sigma$ has constant curvature one, and the horizontal lifts are $X_i=\partial_i-K_i$. All quantities are evaluated in the normal gauges of \eqref{eq:normalgauge}. We use the four-index convention \eqref{eq:fourindex-convention} throughout.

\subsection{Metric jets and the coordinate formula}
The coordinate components are
\[
g_{ij}=h_{ij}+u^2\sigma(K_i,K_j),\quad
g_{ia}=u^2\sigma(K_i,\partial_a),\quad g_{ab}=u^2\sigma_{ab}.
\]
Besides the normal-coordinate two-jets of $h$ and $\sigma$, their nonzero derivatives needed for curvature are
\begin{align}
g_{ia,j}&=u^2L_{ij}^a,& g_{ab,i}&=2uu_i\delta_{ab},\nonumber\\
g_{ij,kl}-h_{ij,kl}&=u^2\sum_a(L_{ik}^aL_{jl}^a+L_{il}^aL_{jk}^a),\nonumber\\
g_{ia,jk}&=2u(u_jL_{ik}^a+u_kL_{ij}^a)+u^2M_{ijk}^a,\nonumber\\
g_{ia,jb}&=u^2S_{ijba},&g_{ab,ij}&=2(u_i u_j+uu_{ij})\delta_{ab}.
\label{eq:metricjet}
\end{align}
The remaining cross derivatives vanish at the point. In particular,
\begin{gather*}
\Gamma^k_{ij}=\Gamma^a_{ij}=0,\qquad
\Gamma^k_{ia}=-u^2L_{ik}^a,\qquad \Gamma^b_{ia}=u_i\delta_{ab}/u,\\
\Gamma^k_{ab}=-uu_k\delta_{ab},\qquad\Gamma^c_{ab}=0.
\end{gather*}

In coordinate indices, with $\Gamma_{IJ}:=\nabla_{\partial_I}\partial_J$, the curvature is
\begin{equation}\label{eq:coordinateR}
\begin{split}
R_{IJKL}={}&\tfrac12(g_{IL,JK}+g_{JK,IL}-g_{IK,JL}-g_{JL,IK})\\
&+g(\Gamma_{IL},\Gamma_{JK})-g(\Gamma_{IK},\Gamma_{JL}).
\end{split}
\end{equation}
Write $\mathscr D_{IJKL}$ for the second-derivative term in the first line and $\mathscr P_{IJKL}$ for the Christoffel products in the second line. The coordinate vertical vectors have length $u$. Therefore a component with $k$ vertical indices is multiplied by $u^{-k}$ when expressed in $(X_i,E_a)$.

\subsection{The two blocks containing the linear term}
For $HVHV$, direct substitution before vertical normalization gives
\begin{align*}
\mathscr D_{iajb}
&=u^2S_{ijab}-(u_i u_j+uu_{ij})\delta_{ab},\\
\mathscr P_{iajb}
&=u^4\sum_k L_{ik}^bL_{jk}^a+u_i u_j\delta_{ab}.
\end{align*}
For the first identity, the two mixed metric derivatives are
$g_{ib,aj}=u^2S_{ijab}$ and
$g_{aj,ib}=u^2S_{jiba}=u^2S_{ijab}$, whereas
$g_{ij,ab}=0$ and $g_{ab,ij}=2(u_i u_j+uu_{ij})\delta_{ab}$.
The product $g(\Gamma_{ij},\Gamma_{ab})$ vanishes, and the horizontal and vertical parts of $g(\Gamma_{ib},\Gamma_{aj})$ give the two terms displayed for $\mathscr P$.
Adding and dividing by $u^2$ cancels $u_i u_j/u^2$ and yields
\[
R_{iajb}=-u^{-1}u_{ij}\delta_{ab}+S_{ijab}
+u^2\sum_kL_{ik}^bL_{jk}^a.
\]

For $HHVV$, each of the four second derivatives contributes the same sign after using both skew symmetries of $S$. The two Christoffel products have identical vertical parts, which cancel. Explicitly,
\begin{align*}
\mathscr D_{ijab}&=2u^2S_{ijab},\\
\mathscr P_{ijab}
&=u^4\sum_k\bigl(L_{ik}^bL_{jk}^a-L_{ik}^aL_{jk}^b\bigr).
\end{align*}
Division by $u^2$ proves \eqref{eq:biHHVV}. Thus the coefficients of the unsuppressed linear term are $S$ on $HVHV$ and $2S$ on $HHVV$. The four-form in \eqref{eq:chosenfourform} contributes $-S$ and $+S$ on these respective blocks. This proves both the cancellation in \eqref{eq:modifiedHV} and the coefficient $3S$ in the remaining cross block.

\subsection{The other four blocks}
In the purely horizontal block $\Gamma_{ij}=0$ at the point. The second derivatives of $u^2\sigma(K_i,K_j)$ give
\[
R_{ijkl}-R^h_{ijkl}
=u^2\sum_a(-2L_{ij}^aL_{kl}^a-L_{ik}^aL_{jl}^a+L_{il}^aL_{jk}^a).
\]
In the $HHHV$ block the Christoffel products again vanish. Its coordinate second-derivative contribution is
\[
\mathscr D_{ijka}
=u(u_jL_{ik}^a-u_iL_{jk}^a+2u_kL_{ij}^a)
+\frac{u^2}{2}(M_{ijk}^a-M_{jik}^a),
\]
which, after division by $u$, proves \eqref{eq:biHHHV}.
For $HVVV$, the second-derivative contribution is zero in the chosen gauges and the horizontal Christoffel products give
\[
\mathscr P_{iabc}
=u^3\sum_k u_k(L_{ik}^c\delta_{ab}-L_{ik}^b\delta_{ac}).
\]
For $VVVV$, the intrinsic fiber term and the horizontal Christoffel products are respectively
\begin{align*}
\mathscr D_{abcd}&=u^2(\delta_{ac}\delta_{bd}-\delta_{ad}\delta_{bc}),\\
\mathscr P_{abcd}&=-u^2|du|_h^2(\delta_{ac}\delta_{bd}-\delta_{ad}\delta_{bc}).
\end{align*}
Vertical normalization gives \eqref{eq:biHVVV} and \eqref{eq:biVVVV}, completing the derivation of all six blocks.

\subsection{Comparison with the principal convention}
For a principal $S^3$ connection, in unit-fiber coefficients one has
\[
L_{ij}^a=-\tfrac12\Omega_{ij}^a,\qquad
S_{ijab}=-\tfrac14\langle[e_a,e_b],\Omega_{ij}\rangle.
\]
For the second identity, the curvature field on the group fiber is right invariant: its coefficient derivative is the adjoint derivative in \eqref{eq:Omega-vertical}. Combining that derivative with the Levi-Civita connection of the bi-invariant unit metric gives
$\langle\nabla_{e_a}L_{ij},e_b\rangle
=\tfrac14\langle[e_a,\Omega_{ij}],e_b\rangle$,
which is the displayed value of $S$. Substitution into $HVHV$ reproduces \eqref{eq:mixed1}; substitution into $HHVV$, with the minus sign specified after \eqref{eq:fourindex-convention}, reproduces \eqref{eq:mixed2}. This comparison checks the convention for the terms linear in curvature without changing either estimate.

\section{Finite parameter checks for the two disks}\label{app:parameters}
We verify that the values in \eqref{eq:parameters} satisfy all the conditions used for the two disks. The argument follows the parameter strategy in \cite[Section 8]{DHZ} and uses the fine estimate \eqref{eq:fine}.

\subsection{The southern coefficient matrix}
For the cap under consideration, $0\le\varphi\le A_K$ and $|d\varphi|\le A_K$. In the matrix in the proof of \cref{cor:smallfiber}, replace $D_0$ by $A_K$ and take $\kappa=1$. The limit is
\[
Q_0=\begin{pmatrix}1&-M_1&0\\-M_1&M_1^2+1&0\\0&0&1\end{pmatrix},
\qquad \lambda_{\min}(Q_0)\ge\lambda_0.
\]
Indeed, the first block has determinant one and trace $M_1^2+2$.

From $A_K\le4B_K^2$ and \eqref{eq:parameters}, one has $\varepsilon_K A_K<1$. Thus for $\mathcal{E}=e^{\varepsilon_K\varphi}$,
\[
1\le \mathcal{E}<3,\quad\sqrt{\mathcal{E}}<2,\quad
\mathcal{E}-1\le3\varepsilon_K A_K,\quad
\sqrt{\mathcal{E}}-1\le\varepsilon_K A_K,\quad
1-\mathcal{E}^{-1}\le\varepsilon_K A_K.
\]
Row-wise estimation gives $\|Q_{\varepsilon_K}-Q_0\|_\op\le\max(C_1,C_2,C_3)\sqrt{\varepsilon_K}$, where
\begin{align*}
C_1&=\tfrac94M_0^2+A_K(M_1+2M_0)+\tfrac32(M_0+M_0^2),\\
C_2&=\tfrac14A_K^2+\tfrac32A_K M_0^2
+A_K(M_1+2M_0)+A_K M_0,\\
C_3&=A_K+\tfrac14A_K^2
+\tfrac32(M_0+M_0^2)+A_K M_0.
\end{align*}
Using $M_0,M_1\le B_K$, $M_1+2M_0\le2B_K$, and $B_K\ge1$, one obtains $C_1\le(53/4)B_K^4$, $C_2\le22B_K^4$, and $C_3\le15B_K^4$. A common conservative bound is $32B_K^4$. Moreover $\lambda_0\ge(2B_K^2)^{-1}$, so
\[
\|Q_{\varepsilon_K}-Q_0\|_\op
\le32B_K^4\sqrt{\varepsilon_K}
=\frac1{8B_K^2}\le\frac{\lambda_0}{4}.
\]
Keeping the weaker $Q_{\varepsilon_K}\ge\lambda_0I/2$ and using
$H^2+M^2+V_*^2\ge\varepsilon_K|\xi|^2$ proves \eqref{eq:southlower}.

The sufficient condition in \cite[Proposition 3.1]{HLY}, with its common constant $C=4$, requires
\[
\Lambda\ge64M_0^2+1024\kappa^{-1}(M_1+1)^2+4.
\]
The matrix estimate above justifies the smaller value of $\Lambda$ used here directly from \cref{prop:fine}.

\subsection{The northern scalar bounds}
For $0<s\le1/2$ put $y=1/s\ge2$. The derivative of the smooth step satisfies
\[
\mathfrak h'(s)=\mathfrak h(s)(1-\mathfrak h(s))
\bigl(s^{-2}+(1-s)^{-2}\bigr)
\le2e^{2-y}y^2\le8.
\]
Symmetry gives the same bound on the other half. Hence $M_\chi\le32$.

Since $F_0^2=3/4$ and $A_K\le4B_K^2$, the quadratic constraint in \eqref{eq:deltaconditions} follows from $48/256<1$. The cubic constraint follows from
\[
\delta^3=\frac1{4096B_K^3}
<\frac1{1056B_K^2}
\le\frac1{8A_K F_0(1+M_\chi)}.
\]
Also $\delta\le1/16<F_0/2$. From \eqref{eq:Fimplicit},
\[
\ell<1+\frac{256}{3}B_K^2<128B_K^2,
\qquad
q_0\ell<2^{-8}B_K^{-8}<\frac12.
\]
Finally $\varrho_\partial<3\varepsilon_K$ and $A_K\le4B_K^2$ give
$d^2<12\varepsilon_K^3B_K^4<1$. These checks establish every constraint used in \eqref{eq:rcontrol}.

\section{Reproducible algebraic checks}\label{app:checks}
The supplied supplement contains three verification scripts that use only the Python standard library. They check the framing-matrix algebra, the two smoothing variance identities, and the bilateral curvature-block identities in a formal normal-gauge metric two-jet. They do not certify the topological identifications or the global analytic estimates. Stable equation labels, rather than equation numbers, identify the statements checked.

\subsection{Framing matrices}
The file \texttt{verify\_framing.py} enumerates all $128$ choices of $(U,b)$ in \eqref{eq:framing-blocks}. It checks $N_0^3=0$ and $\Phi^4=I$, computes negative powers using the matrix inverse, and compares the truncated binomial formula with direct powers for $-8\le m\le8$ on every vector of $\mathbb F_2^5$. These tests verify the algebraic consequence of the group structure and action established in \cref{prop:periodicity}. They do not prove the mapping class theorem or the three geometric calibrations in \cref{lem:calibrations}.

\subsection{Smoothing variance identities}
The file \texttt{verify\_gluing\_algebra.py} expands \eqref{eq:gluedetvariance} and the normal matrix identity preceding \eqref{eq:gluenormal} in independent formal variables with rational coefficients. It checks one tangential determinant identity and all $64$ entries of the normal matrix identity in seam dimension eight. The entries of the three symmetric matrices $A$, $B$ and $h_0^{-1}$ remain independent, and matrix multiplication retains its order. No matrix commutativity or numerical sampling is used. All $65$ polynomial residuals must vanish identically.

These checks do not address the distributional derivative, kernel-mass inequality, uniform tensor estimates, or cutoff localization. Those steps are proved in \cref{app:localgluing}. In particular, the polynomial identities alone are not a smoothing theorem.

\subsection{Curvature from a metric two-jet}
The file \texttt{verify\_bilateral\_jets.py} constructs the coordinate metric
$g=h+u^2\sigma(dy+K,dy+K)$ as a formal two-jet, using product-rule differentiation. Its computation of curvature uses only the metric derivatives, the inverse metric at the evaluation point, and \eqref{eq:coordinateR}; it does not use \eqref{eq:bilateralblocks} to construct that curvature. A separate expression for each claimed block is then subtracted. The coefficients are exact rational Laurent polynomials in the positive radius $u$, its first and second derivatives, the normal base metric jets, and the connection jets $L,S,M$, subject to the stated index symmetries. The unit round fiber supplies its normal-coordinate second jet. In horizontal dimension six and vertical dimension three, all $2835$ indexed residuals in the six listed blocks vanish. A further $648$ residuals check the cancellation of $S$ on $HVHV$ and its coefficient $3S$ on $HHVV$ after the four-form modification.

This is an algebraic check of the curvature identities for the specified jets. It does not prove that a global connection has those normal-gauge properties, bound its derivatives, verify the all-plane inequalities or singular-orbit extensions, or identify the smooth sphere. Those remain geometric arguments in the text.

\subsection{Reproduction and scope}
Run the following commands from the supplementary directory with Python~3:
\begin{verbatim}
python3 verify_framing.py
python3 verify_gluing_algebra.py
python3 verify_bilateral_jets.py
\end{verbatim}
Each script terminates with an assertion failure if a checked identity fails, and otherwise prints a JSON result. Alternatively, 	exttt{python3 run\_all.py} runs all three checks and writes the result files and a manifest recording the interpreter version and the script hashes. Assertions must remain enabled; do not run with 	exttt{-O}. The package includes the results and manifest for the supplied run. The accompanying \texttt{README.txt} describes the scope of each check.

The new coordinate-metric check covers the bilateral blocks in \cref{prop:bilateralblocks}, derived in \cref{app:bilateral-derivation}. The principal estimates in \cref{sec:connection}, including their norm constants, are proved analytically and are not separately certified by this script. The residual counts describe exact algebraic identities, not sampled sectional curvatures. Every geometric application also requires the smooth extensions and marked boundary comparisons in the text.

\begin{center}
\small\renewcommand{\arraystretch}{1.15}
\begin{tabularx}{\textwidth}{@{}>{\raggedright\arraybackslash}p{35mm}>{\raggedright\arraybackslash}X >{\raggedright\arraybackslash}p{37mm}@{}}
\toprule
Claim & Mathematical justification & Supplied computer check \\
\midrule
Smooth-type identifications & \Cref{thm:topology,lem:hopfcomparison,prop:mixedclass,lem:gauge-localization,lem:calibrations,cor:rightclass}. & None.\\
Framing periodicity & \Cref{prop:periodicity}. & Exact finite-field matrix algebra.\\
Principal and bilateral curvature & \Cref{prop:fine,prop:bilateralblocks,app:bilateral-derivation,prop:cylinderpositive,app:parameters}. & Exact bilateral metric-jet identities; no global estimate certification.\\
Smoothing variance identities & \Cref{app:localgluing}. & Exact formal-polynomial expansions.\\
Smooth positive metrics & Curvature estimates and boundary comparisons in \cref{thm:positive,prop:mixedmetric,thm:lastpositive}. & Not certified by the algebraic scripts.\\
\bottomrule
\end{tabularx}
\end{center}

\section*{Use of AI-assisted tools}
ChatGPT (OpenAI) was used to assist with manuscript organization and exposition, the checking and presentation of arguments, and the preparation of verification code. This assistance included making the action-norm dependence and parameter choices explicit in \cref{sec:mixedmetric}, organizing the marked twist construction in \cref{thm:lastpositive}, expanding the relative pseudo-isotopy argument in \cref{lem:tube-extension}, and developing and presenting the local convolution proof in \cref{app:localgluing}. The revision also used these tools to clarify the smoothing-space comparison, separate the gauge-localization argument, derive the ordinary Hopf quotient coefficient, and document the scope of the supplied algebraic checks. The present revision also used these tools to expand the comparison-map and framing explanations, derive the curvature blocks from metric two-jets, and prepare their formal-polynomial check. These checks are limited to the statements specified in \cref{app:checks}.

\bigskip
\begingroup
\small
\noindent\begin{minipage}{\textwidth}
\textsc{Zhiqi Chen}\\
School of Mathematics and Statistics, Guangdong University of Technology, Guangzhou 510520, P. R. China\\
\textit{E-mail:} \href{mailto:chenzhiqi@nankai.edu.cn}{\nolinkurl{chenzhiqi@nankai.edu.cn}}
\end{minipage}\par\medskip

\noindent\begin{minipage}{\textwidth}
\textsc{Shaoqiang Deng}\\
School of Mathematical Sciences and LPMC, Nankai University, Tianjin 300071, P. R. China\\
\textit{E-mail:} \href{mailto:dengsq@nankai.edu.cn}{\nolinkurl{dengsq@nankai.edu.cn}}
\end{minipage}\par\medskip

\noindent\begin{minipage}{\textwidth}
\textsc{Junna Ni}\\
School of Mathematical Sciences, South China Normal University, Guangzhou 510631, P. R. China\\
\textit{E-mail:} \href{mailto:nijunna@126.com}{\nolinkurl{nijunna@126.com}}
\end{minipage}\par\medskip

\noindent\begin{minipage}{\textwidth}
\textsc{Hui Zhang}\\
School of Mathematics, Southeast University, Nanjing 211189, P. R. China\\
\textit{E-mail:} \href{mailto:huizhang@mail.nankai.edu.cn}{\nolinkurl{huizhang@mail.nankai.edu.cn}}
\end{minipage}\par
\endgroup

\begin{thebibliography}{99}
\bibitem{BCG}
D. S. Berman, M. Cederwall and T. Schettini Gherardini,
\emph{Curvature of an exotic 7-sphere},
J. Geom. Phys. \textbf{216} (2025), Article 105590.
\href{https://doi.org/10.1016/j.geomphys.2025.105590}{doi:10.1016/j.geomphys.2025.105590}.

\bibitem{BM}
R. G. Bettiol and R. A. E. Mendes,
\emph{Strongly positive curvature},
Ann. Global Anal. Geom. \textbf{53} (2018), 287--309.
\href{https://doi.org/10.1007/s10455-017-9578-9}{doi:10.1007/s10455-017-9578-9}.

\bibitem{BS}
S. Brendle and R. Schoen,
\emph{Manifolds with 1/4-pinched curvature are space forms},
J. Amer. Math. Soc. \textbf{22} (2009), 287--307.
\href{https://doi.org/10.1090/S0894-0347-08-00613-9}{doi:10.1090/S0894-0347-08-00613-9}.

\bibitem{Burghelea}
D. Burghelea,
\emph{On the homotopy type of Diff$(M^n)$ and connected problems},
Ann. Inst. Fourier \textbf{23}, no.~2 (1973), 3--17.
\url{https://www.numdam.org/item/AIF_1973__23_2_3_0/}.

\bibitem{BKKT}
M. Bustamante, M. Krannich, A. Kupers and B. Tshishiku,
\emph{Mapping class groups of exotic tori and actions by $\mathrm{SL}_d(\Z)$},
preprint, \href{https://arxiv.org/abs/2305.08065v2}{arXiv:2305.08065v2} (2024).

\bibitem{CKS}
R. Casals, A. Keating and I. Smith, with an appendix by S. Courte,
\emph{Symplectomorphisms of exotic discs},
J. \'Ec. polytech. Math. \textbf{5} (2018), 289--316.
\href{https://doi.org/10.5802/jep.71}{doi:10.5802/jep.71}.

\bibitem{CDZ}
Z. Chen, S. Deng and H. Zhang,
\emph{Positive sectional curvature on a family of exotic fifteen-spheres},
unpublished manuscript, 2 October 2026.

\bibitem{CH}
D. Crowley and I. Hambleton,
\emph{Finite group actions on Kervaire manifolds},
Adv. Math. \textbf{283} (2015), 88--129.
Preprint: \href{https://arxiv.org/abs/1305.6546v2}{arXiv:1305.6546v2}.

\bibitem{CS}
D. Crowley and T. Schick,
\emph{The Gromoll filtration, $KO$-characteristic classes and metrics of positive scalar curvature},
Geom. Topol. \textbf{17} (2013), 1773--1789.
\href{https://doi.org/10.2140/gt.2013.17.1773}{doi:10.2140/gt.2013.17.1773}.

\bibitem{DHZ}
S. Deng, Z. Hu and H. Zhang,
\emph{Positive sectional curvature and non-isometric circle actions on a family of eleven-spheres},
preprint, \href{https://arxiv.org/abs/2609.32680v2}{arXiv:2609.32680v2}, 5 October 2026.

\bibitem{GG}
F. Galaz-Garc\'ia,
\emph{Positive sectional curvature on the exotic 8-sphere and order-three homotopy 10-spheres},
preprint, \href{https://arxiv.org/abs/2609.38126v2}{arXiv:2609.38126v2}, 1 October 2026.

\bibitem{GKS}
S. Goette, M. Kerin and K. Shankar,
\emph{Highly connected 7-manifolds and non-negative sectional curvature},
Ann. of Math. (2) \textbf{191} (2020), 829--892.
\href{https://doi.org/10.4007/annals.2020.191.3.3}{doi:10.4007/annals.2020.191.3.3}.

\bibitem{GM}
D. Gromoll and W. Meyer,
\emph{An exotic sphere with nonnegative sectional curvature},
Ann. of Math. (2) \textbf{100} (1974), 401--406.
\href{https://doi.org/10.2307/1971078}{doi:10.2307/1971078}.

\bibitem{GVWZ}
K. Grove, L. Verdiani, B. Wilking and W. Ziller,
\emph{Non-negative curvature obstructions in cohomogeneity one and the Kervaire spheres},
Ann. Sc. Norm. Super. Pisa Cl. Sci. (5) \textbf{5} (2006), 159--170.
\url{https://www.numdam.org/item/ASNSP_2006_5_5_2_159_0/}.

\bibitem{GZ}
K. Grove and W. Ziller,
\emph{Curvature and symmetry of Milnor spheres},
Ann. of Math. (2) \textbf{152} (2000), 331--367.
\href{https://doi.org/10.2307/2661385}{doi:10.2307/2661385}.

\bibitem{HLY}
Y.-H. He, Z. Liu and S.-T. Yau,
\emph{Positive Sectional Curvature on All Smooth 7-spheres},
preprint, \href{https://arxiv.org/abs/2609.29426v1}{arXiv:2609.29426v1}, 24 September 2026.

\bibitem{Hitchin}
N. Hitchin,
\emph{Harmonic spinors},
Adv. Math. \textbf{14} (1974), 1--55.

\bibitem{KM}
M. A. Kervaire and J. W. Milnor,
\emph{Groups of homotopy spheres. I},
Ann. of Math. (2) \textbf{77} (1963), 504--537.
\href{https://doi.org/10.2307/1970128}{doi:10.2307/1970128}.

\bibitem{KK}
M. Krannich and A. Kupers,
\emph{Pontryagin--Weiss classes and a rational decomposition of spaces of homeomorphisms},
preprint, \href{https://arxiv.org/abs/2504.20265v2}{arXiv:2504.20265v2} (2025).

\bibitem{KS}
R. C. Kirby and L. C. Siebenmann,
\emph{Foundational essays on topological manifolds, smoothings, and triangulations},
Annals of Mathematics Studies, No.~88, Princeton University Press, Princeton, N.J., 1977.
With notes by John Milnor and Michael Atiyah.

\bibitem{Krylov}
N. A. Krylov,
\emph{Relative mapping class group of $S^p\times D^q$},
preprint, \href{https://arxiv.org/abs/math/0409342v1}{arXiv:math/0409342v1} (2004).

\bibitem{Milnor}
J. Milnor,
\emph{On manifolds homeomorphic to the 7-sphere},
Ann. of Math. (2) \textbf{64} (1956), 399--405.
\href{https://doi.org/10.2307/1969983}{doi:10.2307/1969983}.

\bibitem{ONeill}
B. O'Neill,
\emph{The fundamental equations of a submersion},
Michigan Math. J. \textbf{13} (1966), 459--469.
\href{https://doi.org/10.1307/mmj/1028999604}{doi:10.1307/mmj/1028999604}.

\bibitem{Puttmann}
T. P\"uttmann,
\emph{On the commutator of unit quaternions and the numbers 12 and 24},
preprint, \href{https://arxiv.org/abs/1101.5147v4}{arXiv:1101.5147v4} (2012).

\bibitem{RW}
P. Reiser and D. J. Wraith,
\emph{A generalization of the Perelman gluing theorem and applications},
preprint, \href{https://arxiv.org/abs/2308.06996v2}{arXiv:2308.06996v2}, 17 January 2024.

\bibitem{Sarkar}
A. Sarkar,
\emph{Smooth structures on the product of 3-connected 8-manifolds with spheres},
Proc. Indian Acad. Sci. Math. Sci. \textbf{135} (2025), Article 20.
\href{https://doi.org/10.1007/s12044-025-00829-2}{doi:10.1007/s12044-025-00829-2}.
Preprint: \href{https://arxiv.org/abs/2502.20736}{arXiv:2502.20736}.

\bibitem{Speranca}
L. D. Speran\c ca,
\emph{Pulling back the Gromoll--Meyer construction and models of exotic spheres},
Proc. Amer. Math. Soc. \textbf{144} (2016), 3181--3196.
\href{https://doi.org/10.1090/proc/12945}{doi:10.1090/proc/12945}.

\bibitem{Stolz}
S. Stolz,
\emph{Simply connected manifolds of positive scalar curvature},
Ann. of Math. (2) \textbf{136} (1992), 511--540.
\href{https://doi.org/10.2307/2946598}{doi:10.2307/2946598}.

\bibitem{Totaro}
B. Totaro,
\emph{Complexifications of nonnegatively curved manifolds},
J. Eur. Math. Soc. \textbf{5} (2003), 69--94.
\href{https://doi.org/10.1007/s10097-002-0044-y}{doi:10.1007/s10097-002-0044-y}.

\bibitem{Wang}
W. Wang,
\emph{On some Gromoll filtration groups and fundamental groups of $\Diff_\partial(D^n)$},
preprint, \href{https://arxiv.org/abs/2504.04875v2}{arXiv:2504.04875v2} (2026).
\end{thebibliography}
\end{document}